\documentclass[10pt]{amsart}

	\usepackage{amscd} 
	\usepackage{amssymb}
	\usepackage{setspace}	
	\usepackage{fancyhdr}	
	\usepackage{indentfirst} 	
	\usepackage{amsthm}	
	\usepackage{graphicx}	
	\usepackage{amssymb}
	\usepackage{mathrsfs}

	\theoremstyle{definition}
		\newtheorem{thm}{Theorem}[section]
		\newtheorem{lem}[thm]{Lemma}
		
		\newtheorem{prop}[thm]{Proposition}
		\newtheorem{cor}[thm]{Corollary}

		\newtheorem{defn}[thm]{Definition}

\newcommand{\K}{\mathbb{C}}
\newcommand{\KO}{\mathbb{K}}

\newcommand{\N}{\mathbb{N}}
\newcommand{\R}{\mathbb{R}}

\newcommand{\Z}{\mathbb{Z}}

\newcommand{\Ad}{\text{\rm Ad}}

\newcommand{\Cstar}{{\text\rm C}^*}

\newcommand{\Sp}{\text{\rm sp}}

\newcommand{\Prim}{\text{\rm Prim}}

\newcommand{\rank}{\text{\rm rank}}
\newcommand{\Ima}{\text{\rm Im }}

\newcommand{\id}{\text{\rm id}}

\newcommand{\supp}{\text{\rm supp }}

\title{STABLE RECURSIVE SUBHOMOGENEOUS ALGEBRAS}
\thanks{*Address: \emph{Department of Mathematics, East China Normal University, North Zhongshan Road 3663,
Shanghai, China, 200062}
}

\begin{document}
\thispagestyle{empty}

\maketitle

\begin{center}
Hutian Liang*
\end{center}

\vspace{0.3in}
\begin{center}
{\bf Abstract}
\end{center}

{\small In this paper, we introduce \emph{stable recursive subhomogeneous algebras} 
(SRSHAs), which is analogous to recursive subhomogeneous algebras (RSHAs) 
introduced by N.~C.~Phillips in the studies of free minimal integer actions on compact
metric spaces.
The difference between the stable version and the none stable version is that the irreducible 
representations of SRSHAs are  infinite dimensional, but the irreducible representations of the 
RSHAs are finite dimensional.  While RSHAs play an important role in the study of free minimal
integer actions on compact metric spaces, SRSHAs play an analogous role in the study of free
minimal actions by the group of the real numbers on compact metric spaces. 
In this paper, we show that simple inductive limits of 
SRSHAs with no dimension growth in which the connecting maps are injective and non-vanishing have 
topological stable rank one.
}

\section{Introduction}
Recursive subhomogeneous algebras, abbreviated RSHA, are introduced by N. C. Phillips
in \cite{PhillipsRSHA}.  Essentially, a RSHA is an iterated
pull back of algebras of the form $C(X, M_n),$ where the spaces $X$ are taken to be
compact Hausdorff space,  $M_n$ is the algebra of
$n\times n$-matrices, and $C(X, M_n)$ is the algebra of all continuous functions from
$X$ into $M_n.$  In some sense, a recursive subhomogeneous
algebra is formed by ``gluing'' finitely many algebras of the form $C(X, M_n)$
together.  RSHAs played a crucial role in the study of free minimal $\Z$ actions on compact
metric spaces of finite dimension, where $\Z$ denotes the group of integers.  
In \cite{HLinPhillips}, H. Lin and N. C. Phillips showed that
under certain hypothesis about traces, the crossed product obtained from a free minimal
$\Z$ action on a finite dimensional compact metric space has tracial rank zero.  
The proof of this result relies heavily on the fact the crossed product algebra contains a 
subalgebra that can be written as a simple direct limit of RSHAs, whose structure is simple 
enough that it is possible to show that the RSHA has tracial rank zero.  Of course the other 
important part of  the proof is that there is a link between the subalgebra and the crossed 
product algebra
 so that the property of the subalgebra can be extended to the entire crossed product. 

While RSHAs are important tools for the study of free minimal $\Z$ actions on compact metric
spaces, they cannot be applied to the study of the free minimal $\R$ actions, where $\R$
stands for the group of the real numbvers.  This because $\R$ is not discrete.  In the cases of 
$\R$ actions, we need a ``stable'' version of the RSHAs. In this paper, we introduce an 
analogous ``stable'' version of RSHA. We will use $\KO$ to denote the algebra of all compact 
operators on the separable infinite dimensional Hilbert space throughout the paper.  If $A$ is any
$\Cstar$-algebra, we will take $C(\varnothing, A)$ to be the zero algebra.

\begin{defn}
\label{non-vanishingHom}
Let $A, B$ be $\Cstar$-algebras,  let $X$ be a compact Hausdorff space,
and let \mbox{$\phi\colon A\rightarrow C(X, B)$} be a *-homomorphism.  
We say $\phi$ is
{\emph{non-vanishing}} if for all $x\in X,$ there exists some $a\in A$ such
that $\phi(a)(x)\neq 0.$  
\end{defn}

Note that in the above definition, if $X=\varnothing,$ then $\phi$ is vacuously
non-vanishing.

\begin{defn}
\label{RSHASystem}
Let $H$ be a separable infinite dimensional Hilbert space and let $\KO$ denote
the set of all compact operators on $H$.
Let $n$ be a positive integer, 
let
$X_1, \ldots, X_n$ be compact Hausdorff spaces,  let 
$X_k^{(0)}\subseteq X_k$ be closed subspaces for $k=2, \ldots, n,$ 
and let $R_k\colon C(X_k, \KO)\rightarrow C(X_k^{(0)}, \KO)$
be the restriction map for $k=2, \ldots, n.$
For each $k$ with $2\leq k\leq n,$
let $\phi_k\colon A^{(k-1)}\rightarrow C(X_k^{(0)}, \KO)$ be a non-vanishing
*-homomorphism,  let $A^{(1)}=C(X_1, \KO),$ and inductively define 
$$A^{(k)}=\{(a, b)\in A^{(k-1)}\oplus C(X_k, \KO)\colon 
\phi_k(a)=R_k(b)\}.$$  We call 
$$\left(X_1, A^{(1)}, \left(X_k, X_k^{(0)}, \phi_k,
R_k, A^{(k)}\right)_{k=2}^n\right)$$
a {\emph{stable recursive sub-homogeneous system}},
abbreviated SRSH system, 
and 
call the algebra $A^{(n)}$ the {\emph{stable recursive sub-homogeneous
algebra}}, 
abbreviated by SRSHA, corresponding to the system.  

Let $A$ be a $\Cstar$-algebra.  We say that $A$ has a {\emph{stable recursive
sub-homogeneous decomposition}} if there exists a stable recursive
sub-homogeneous system $$\left(X_1, A^{(1)}, \left(X_k, X_k^{(0)}, \phi_k, R_k,
A^{(k)}\right)_{k=2}^n\right)$$ such that $A\cong A^{(n)},$ in which case we also say
that $A$ is a stable recursive sub-homogeneous algebra, and call the system
a stable recursive sub-homogeneous decomposition of $A.$

The integer $n$ is called the {\emph{length}} of the 
system (or the decomposition). The spaces $X_1, \ldots, X_n$ are
called the {\emph{bases spaces}} of the system.
The space $X=\bigsqcup_{k=1}^n X_k$ is called the {\emph{total space}} of the
system.
The spaces $X_2^{(0)}, \ldots, X_n^{(0)}$ are called the {\emph{attaching spaces}}
of the system.
The maps $R_2, \ldots, R_k$ are called the {\emph{restriction maps}} of the system.
The maps $\phi_2, \phi_3, \ldots, \phi_n$ are called the
{\emph{attaching map}} of the system.
For each $k \in \{1, \ldots, n\},$ the algebra $A^{(k)}$ is
called {\emph{$k$-th partial algebra}} of the system.
\end{defn}

Note that a SRSH system of length 1 is simply $(X_1, C(X_1, \KO)).$
For a SRSHA $A,$ the decomposition is by no means
unique.  We allow any or all of the attaching spaces to be the empty set.  If
$X_k^{(0)}=\varnothing$ for
some $k,$ then $A^{(k)}$ is simply $A^{(k-1)}\oplus
C(X_k, \KO).$  
If $A$ has a stable SRSH decomposition 
$$\left(X_1, A^{(1)}, \left(X_k, X_k^{(0)}, \phi_k, R_k, A^{(k)}\right)_{k=2}^n\right),$$ then $A$ is a
$\Cstar$-subalgebra of $\bigoplus_{k=1}^n C(X_k, \KO);$
also for each $k \in \{1, \ldots, n\},$ the $k$-th partial algebra is also a SRSHA with the
decomposition being
$$\left(X_1, A^{(1)}, \left(X_i, X_i^{(0)}, \phi_i, R_i,
A^{(i)}\right)_{i=2}^k\right).$$
Let $a=(a_1, \ldots, a_n)\in A$ and  let $x$ be in the total space
$X.$
Then there exists unique $k$ such that $x\in X_k.$  We will use $a(x)$ to
denote $a_k(x).$ So for each $x\in X,$ the map $A\rightarrow \KO$
sending $a\mapsto a(x)$ is a clearly *-homomorphism.  If $1\leq k\leq l\leq n,$ 
then it is easily verified that the map $p_{l, k}\colon A^{(l)}\rightarrow A^{(k)}$
defined
by $p_{l, k}(a_1, \ldots, a_l)=
(a_1, \ldots, a_k)$ is a surjective *-homomorphism.
If $1\leq
k\leq l \leq m\leq n,$ then $p_{m, k}=p_{l, k}\circ p_{m, l}.$

In this paper will establish the following result about simple inductive limits of stable 
recursive subhomogeneous algebras.

\begin{defn}
Let $A$ be $\Cstar$-algebra.  We say that $A$ has \emph{topological stable rank
 one} (or simply \emph{stable rank one}) if the set of invertible elements of $A$ is 
norm dense in $A$.
\end{defn}
\begin{thm}
\label{StableRankOfSimpleInductLimOfSRSHA}
Let $(A_n, \psi_n)$ be an inductive system of SRSHAs and let $A$ be the
inductive limit.  Let $X_n$ be the total space for $A_n.$
Suppose that $\psi_n$ is injective and non-vanishing for all $n,$ 
and suppose that $A$ is
simple.  Also assume that there exists $d\in \N$ such that 
$\dim (X_n)\leq d$ for all $n\geq 1.$
Then $A$ has topological stable rank one.
\end{thm}

\section{Ideals and Homomorphisms of SRSHAs}

In this section we establish some results about the spectrum, primitive ideal
space, and ideals of a SRSHA.  
We will use
$\widehat A$ to denote the spectrum of $A,$ i.e.\ the space of all irreducible
representations of $A,$ and if $\pi$ is an irreducible representation of $A,$ we
will use $[\pi]$ to denote the corresponding element in $\widehat A.$ We will use
$\Prim(A)$ to denote the primitive ideal space
of $A.$
The next lemma is a standard result.
\begin{lem}
\label{IrrRepC(X,K)}
Let $X$ be a locally compact Hausdorff space and let $A=C_0(X, \KO).$  For
each
$x\in X,$ let ${\mathrm{ev}}_x\colon A\rightarrow \KO$ be defined by ${\mathrm{ev}}_x(f)=f(x).$  
Then
\begin{enumerate}
\item 
the map $X\rightarrow \widehat{A}$ defined by $x\mapsto [{\mathrm{ev}}_x]$ is a well
defined bijection;

\item the map $X\rightarrow \Prim(A)$ defined by $x\mapsto \{f\in
A\colon f(x)=0\}$ is a well-defined bijection.
\end{enumerate}
\end{lem}

\begin{lem}
\label{IrreRepSRSHA}
Let $n$ be a positive integer.  Let $$(X_1, A^{(1)}, (X_k, X_k^{(0)},
\psi_k, R_k, A^{(k)})_{k=2}^n)$$ be a stable recursive sub-homogeneous
system and let $A=A^{(n)}.$  Let $X_1^{(0)}=\varnothing.$  Then
\begin{enumerate}
\item the map 
$M\colon\bigsqcup_{k=1}^n (X_k\setminus X_k^{(0)})\rightarrow \Prim(A)$ 
defined by $M(x)=\{a\in A\colon
a(x)=0\}$ is a well defined bijection.
\item for each $x\in \bigsqcup_{k=1}^n (X_k\setminus X_k^{(0)}),$ the evaluation map
${\mathrm{ev}}_x\colon A\rightarrow \KO,$ given by $a\mapsto a(x),$ is non-zero; also the map 
$S\colon\bigsqcup_{k=1}^n (X_k\setminus X_k^{(0)})\rightarrow \widehat A$ 
defined by $S(x)=[{\mathrm{ev}}_x]$ is a well defined bijection.
\end{enumerate}
\end{lem}

\begin{proof}
Induct on $n.$  The case when $n=1$ is given by Lemma
\ref{IrrRepC(X,K)}.
Suppose that statement holds for some $n,$  let 
\vskip 1em plus 5pt minus 3pt
$$\left(X_1, A^{(1)}, \left(X_k, X_k^{(0)}, \psi_k, R_k, A^{(k)}\right)_{k=2}^{n+1}\right)$$ 
be a SRSH system of length
$n+1$ and let $A=A^{(n+1)}.$

 Let $1\leq i\leq n+1$ and let $x\in X_i\setminus X_i^{(0)}.$
 Define
 $\pi\colon A^{(n+1)}\rightarrow \KO$ by $\pi(f_1, \ldots, f_{n+1})=f_i(x).$
 Then $\pi$ is a clearly a *-homomorphism. Let $a\in \KO.$
Choose $h\in C(X_i)$ such that  $h(x)=1$ and $\supp\,
 h\subseteq X_i\setminus X_i^{(0)},$ and let $f\in C(X_i, \KO)$ be defined by
 $f(y)=h(y)a.$  
 Then $\supp\, f\subseteq X_i\setminus X_i^{(0)}.$  
Hence 
 $R_i(f)=f|_{X_i^{(0)}}=0=\psi_i(0),$ and so $(0, 
 \ldots, 0, f)\in A^{(i)}.$  Since the map $A^{(n+1)}\rightarrow
 A^{(i)}$ defined by $(g_1, \ldots, g_{n+1})\mapsto (g_1, \ldots, g_i)$ is
 surjective, there exist $g_{i+1}, \ldots, g_{n+1}$ such that 
$\xi=(0, \ldots,
 0, f, g_{i+1}, \ldots, g_{n+1})\in A^{(n+1)}.$
 Then $\pi(\xi)=f(x)=a.$  Thus $\pi={\mathrm{ev}}_x$ maps onto $\KO,$ 
 and so $\pi$ is non-zero and irreducible.  This shows that the map $S$
defined in part 2
 of the statement of the lemma is well defined.  Further, this also shows that
 $$\{(g_1, \ldots, g_{n+1})\in A^{(n+1)}\colon g_i(x)=0\}=\ker \pi\in
\text{Prim}(A^{(n+1)}),$$ and so $M$ defined in part 1 of the statement of the
lemma is well defined.  

Now consider $$I_{n+1}=\{(f_1,
\ldots, f_n, f_{n+1})\in A^{(n+1)}\colon (f_1, \ldots, f_n)=0\}.$$  Then it is
clear that $I_{n+1}$ is a closed two sided ideal of $A.$  Note that
if $(f_1, \ldots, f_{n+1})\in I_{n+1},$ then $0=\psi_{n+1}(f_1, \ldots,
f_n)=R_{n+1}(f_{n+1}),$ and so $f_{n+1}$ vanishes on $X_{n+1}^{(0)}.$ Define
$$\phi\colon I_{n+1}\rightarrow C_0(X_{n+1}\setminus X_{n+1}^{(0)}, \KO)$$ by
$\phi(f_1, \ldots, f_{n+1})=f_{n+1}|_{{X_{n+1}}\setminus X_{n+1}^{(0)}}.$
This map
is well defined because if $(f_1, \ldots, f_{n+1})\in I_{n+1},$ then
$f_{n+1}$ vanishes on $X_{n+1}^{(0)},$ so $f_{n+1}\in C_0(X_{n+1}\setminus
X_{n+1}^{(0)}, \KO).$  Then it is clear that $\phi$ is a *-isomorphism.

Now let $\pi\colon A\rightarrow B(H)$
be a non-zero irreducible representation.  First
assume that $\pi|_{I_{n+1}}\colon I_{n+1}\rightarrow B(H)$ is not the zero
representation.  Then $\pi|_{In_{n+1}}$ is also irreducible.  Thus $\pi\circ
\phi^{-1}$ is an irreducible representation of $C_0(X_{n+1}\setminus
X_{n+1}^{(0)}, \KO),$ and so by Lemma \ref{IrrRepC(X,K)} there exists $x\in
X_{n+1}\setminus
X_{n+1}^{(0)},$
such that $[\pi\circ\phi^{-1}]=[{\mathrm{ev}}_x].$  Then there exists a unitary $u$
such that $\pi\circ\phi^{-1}=\Ad(u)\circ {\mathrm{ev}}_x,$ where $
\Ad(u)\colon\KO\rightarrow\KO$
is defined by $\Ad(u)(a)=uau^*.$  Define $\pi'\colon A\rightarrow
B(H)$ by
$\pi'(f_1, \ldots, f_{n+1})=\Ad(u)(f_{n+1}(x)).$  Then
$\pi|_{I_{n+1}}=\pi'|_{I_{n+1}}.$  Since $\pi|_{I_{n+1}}=\pi'|_{I_{n+1}}$
is irreducible, hence non-degenerate, we have  $\pi=\pi'.$  Then $S(x)
=[\pi']=[\pi].$

Now suppose that $\pi|_{I_{n+1}}=0.$  
Define $\psi\colon A^{(n+1)}\rightarrow A^{(n)}$ by $\psi(f_1, \ldots,
f_{n+1})=(f_1, \ldots, f_n).$
Consider the short exact sequence
$$ 0\rightarrow I_{n+1}\rightarrow A^{(n+1)}\stackrel{\psi}{\longrightarrow}
A^{(n)}\rightarrow
0.$$
  Since $\pi$
restricts to zero on $I_{n+1},$ $\pi$ factors through $A^{(n)}.$  That is, there
exists $\widetilde\pi\colon A^{(n)}\rightarrow B(H)$ such that  $\widetilde\pi\circ
\psi=\pi.$  Then $\Ima \pi=\Ima \widetilde\pi.$  Since $\pi$ is irreducible, we
see that $\widetilde\pi$ is also irreducible.  Thus by the inductive hypothesis,
we see that there exists some $1\leq i\leq n$ and some $x\in X_i\setminus
X_i^{(0)}$ such that  $[\widetilde \pi]=[{\mathrm{ev}}_x].$  So there exists a unitary such
that $\widetilde\pi(f)=\Ad(u)(f(x))$ for all $f\in A^{(n)}.$  Then for all $f=(f_1,
\ldots, f_n, f_{n+1})\in
A^{(n+1)},$ we have $\pi(f)=\widetilde\pi(\psi(f))=\widetilde\pi(f_1, \ldots,
f_n)=\Ad(u)(f_i(x)).$  Thus $[\pi]=S(x),$ and hence $S$ is surjective.
If $J\in \Prim(A),$ then there exists some irreducible representation
$\pi$ of $A$ such that $J=\ker \pi.$  So there exists $x\in \bigsqcup_{k=1}^{n+1}
\left(X_k\setminus X_k^{(0)}\right)$ such that $[{\mathrm{ev}}_x]=[\pi].$  It follows
that
$$J=\ker\pi=\ker
{\mathrm{ev}}_x=\{a\in A\colon  a(x)=0\}=M(x).$$ Thus $M$ is also surjective.

Next we show that $M$ and $S$ are injective.  Let $x, y\in \bigsqcup_{k=1}^{n+1}
(X_k\setminus X_k^{(0)})$ and suppose that $x\neq y.$  First assume that there exist
$1\leq j<k\leq n$ such that $x\in X_j\setminus X_j^{(0)}$ and  $y\in
X_k\setminus X_k^{(0)}.$  Let $h\in C(X_k)$ satisfy 
$h(y)=1$ and $\supp h\subseteq X_k\setminus X_k^{(0)},$ let $a\in \KO$
be a non-zero element,  let $f=ah,$ and let $b=(0,
\ldots, 0, f)\in A^{(k)}.$  Let $f_{k+1}, \ldots, f_{n+1}$ be such that $g=(b,
f_{k+1}, \ldots, f_{n+1})\in A^{(n+1)}.$  Then $g(x)=0,$ but $g(y)=a\neq 0.$
Thus $g\in M(x),$ but $g\notin M(y),$ and so $M(x)\neq M(y).$  Since $M(x)=\ker
{\mathrm{ev}}_x$ and $M(y)=\ker {\mathrm{ev}}_y,$ we have $S(y)=[{\mathrm{ev}}_y]
\neq [{\mathrm{ev}}_x]=S(x).$  
Now suppose that $x, y\in X_k\setminus X_k^{(0)}$ for some $1\leq k\leq n.$
Since $x, y$ are different, there exists an open $U\subseteq X_k\setminus
X_k^{(0)}$ such that $y\in U,$ but $x\notin U.$  Choose $h\in C(X_k)$
such that $h(y)=1$ and $h$ vanishes outside of $U.$  Let $a\in
\KO$ be non-zero.  Let $f=ah.$ Then $f$ vanishes on $X_k^{(0)}$.  So there exist
$$g_{k+1}\in C(X_{k+1}, \KO), \ldots, g_{n+1}\in C(X_{n+1}, \KO)$$ such that
$g=(0, \ldots, 0, f, g_{k+1}, \ldots, g_{n+1})$ belongs to $A$.
Then $g(x)=f(x)=0$ and $g(y)$$=f(y)=a.$  It follows that 
$g\in M(x),$ but $g\notin M(y).$
So $M(y)\neq M(x),$ and consequently $S(x)\neq S(y).$
\end{proof}

\begin{cor}
\label{SRSHASeparatesPoints}
Let $$\left(X_1, A^{(1)}, \left(X_k, X_k^{(0)},
\psi_k, R_k, A^{(k)}\right)_{k=2}^n\right)$$ be a stable recursive sub-homogeneous
system and let $A=A^{(n)}.$  Let $X_1^{(0)}=\varnothing.$ 
Then for all $x, y\in \bigsqcup_{k=1}^n (X_k\setminus X_k^{(0)})$ with
$x\neq y,$ there exist some $a, b\in A$ such that $a(x)=0,$ $a(y)\neq
0,$ $b(x)\neq 0,$ and $b(y)=0.$
\end{cor}

\begin{proof}
First suppose that $x\in X_j\setminus X_j^{(0)}$ and $y\in X_k\setminus
X_k^{(0)},$ where $1\leq j<k\leq n.$  Then the element $a\in A$
needed is constructed in the last paragraph of the proof of \ref{IrreRepSRSHA}.  Next
we construct the element $b.$  Let
$h\in C(X_j)$ be such that $h(x)=1$ and $h$ vanishes on $X_j^{(0)},$  let
$\xi\in \KO$ be non-zero, and let $f=h\xi.$  Then $(0, \ldots, 0, f)\in
A^{(j)}.$  Choose $b'\in A^{(k-1)}$ such that the first $j$ entries of $b'$ are $(0,
\ldots, 0, f).$ Let $c=\phi_k(b').$  Let $V$ be an open neighborhood of
$X_k^{(0)}$ that does not contain $y,$ and choose $h'\in C(X_k)$ such that
$h'|_{X_k^{(0)}}=1$ and $h'$ vanishes outside of $V.$  Let $c'$ be any
extension of $c$ over $X_k,$ and let $f'=h'c'.$  Then
$f'|_{X_k^{(0)}}=c=\phi_k(b').$  So $(b', f')\in A^{(k)}.$  Choose $b\in A$
such that the first $k$ entries of $b$ are $(b', f').$  Then
$b(x)=f(x)=\xi\neq 0,$ and $b(y)=f'(y)=h'(y)c'(y)=0.$  

Now suppose that $x, y\in X_k\setminus X_k^{(0)}.$  Let $U_x$ and $U_y$ be
two disjoint open sets contained in $X_k\setminus X_k^{(0)}$ such that
$x\in U_x$ and $y\in U_y.$  Choose $h_x\in C(X_k)$ and $h_y\in C(X_k)$ such
that $h_x(x)=1$ and $h_y(y)=1,$ $h_x$ vanishes outside of $U_x,$ and 
$h_y$ vanishes outside of $U_y.$  Let $\xi\in \KO$ be non-zero.  Let
$f_x=h_x\xi,$ and $f_y=h_y\xi.$  Then $a'=(0, \ldots, f_y)\in A^{(k)}$ and
$b'=(0, \ldots, 0, f_x)\in A^{(k)}.$  Let $a, b\in A$ be such that the
first $k$ entries of $a$ and $b$ are, respectively, $a'$ and $b'.$  Then
\begin{align*}
&a(x)=a'(x)=f_y(x)=0,\\
&a(y)=a'(y)=f_y(y)=\xi\neq 0,\\
&b(x)=b'(x)=f_x(x)=\xi\neq 0,\\
&b(y)=b'(y)=f_x(y)=0.
\end{align*}
\end{proof}

\begin{cor}
\label{SRSHAIdealClosedSet}
Let $n$ be a positive integer.  Let $$\left(X_1, A^{(1)}, \left(X_k, X_k^{(0)},
\psi_k, R_k, A^{(k)}\right)_{k=2}^n\right)$$ be a stable recursive sub-homogeneous
system, and let $A=A^{(n)}.$  Let $X_1^{(0)}=\varnothing.$  Let $I\subseteq A$ be
a closed two sided ideal of $A.$  Then there exists a closed set $F\subseteq
X=\bigsqcup_{k=1}^n X_k$ such that $I=\{a\in A\colon  a|_F=0\}.$
\end{cor}

\begin{proof}
Let $I$ be a closed two sided ideal of $A.$   If $I=0,$ then take
$F=X.$ If $I=A,$ then take $F=\varnothing.$
Now assume that $I$ is proper and non-zero. Recall that for any $\Cstar$-algebra $B$ 
and for any closed two sided ideal $I$ of $B$, the hull of $I$, denoted by $\text{hull}(I)$, 
is the set of all primitive ideals of $B$ that contain $I$; and for any subset $S\subseteq \Prim(B)$, 
the kernel of $S$, denoted by $\ker(S)$ is the intersection of all the members of $S$.
We know that $I=\ker(\text{hull}(I)).$  
Let $M$ be as in Lemma \ref{IrreRepSRSHA}.
Let 
$F=\overline{M^{-1}(\text{hull}(I))}.$
We will verify that $I=\{a\in A\colon a|_F=0\}.$
Let $J$ denote $\{a\in A\colon a|_F=0\}.$

Let $a\in I,$ and let $x\in
M^{-1}(\text{hull}(I)).$  Then $M(x)\in \text{hull}(I),$ and so
$a\in I\subseteq M(x).$  So $a(x)=0.$  This holds for all $x\in
M^{-1}(\text{hull}(I)).$  Thus $a$ vanishes on
$M^{-1}(\text{hull}(I)).$  Since $a$ is continuous,
$a|_F=0.$  So $a\in J,$ and so 
$I\subseteq J.$ 
Now suppose that $a\in J.$  Let $L\in \text{hull}(I).$
Then there exists $x\in X$ such that
$L=M(x),$ and so $x\in
M^{-1}(\text{hull}(I))\subseteq F.$  The condition $a\in
J$ implies that $a(x)=0,$ which implies that $a\in
M(x)=L.$  This holds for all $L\in \text{hull}(I),$ so 
$a\in \ker(\text{hull}(I))=I.$  Thus $J\subseteq I,$ and so $I=J.$
\end{proof}

The next theorem is a restatement of Theorem 1.4.4 in \cite{ArvesonText}.
\begin{thm}
\label{RepSubAlgOfK}
Let $H$ be an arbitrary Hilbert space, and let $A\subseteq K(H)$ be a
non-zero $\Cstar$-subalgebra.  Then there exists an index set $I$ and a
family $(p_i)_{i\in I}$ of mutually orthogonal projections in $B(H),$
indexed by $I$, such that
\begin{enumerate}
\item $p_i\in A'$ for all $i\in I,$ where $A'$ denotes the commutant of $A;$
\item $p_iAp_i=K(p_iH)$ for all $i\in I$ (we identify $K(p_iH)$  with
$p_iK(H)p_i$ in an obvious way);
\item $\|a\|=\sup_{i\in I} \|p_iap_i\|$ for all $a\in A;$
\item $\sum_{i\in I} p_iap_i$ converges to $a$ in norm for all $a\in A;$
\item for all $a\in A$ and for all $\epsilon>0,$ there exists a finite
subset $F\subseteq I$ such that $\|p_iap_i\|<\epsilon$ for all $i\notin F.$
\end{enumerate}
\end{thm}

\begin{prop}
\label{CharaterOfKMapsFromSRSHA}
Let $H$ be a separable infinite dimension Hilbert space and let $\KO$ denote the 
set of all compact operators on $H$.
Let 
$$\left(X_1, A^{(1)}, \left(X_k, X_k^{(0)}, \phi_k, R_k, A^{(k)}\right)_{k=2}^n\right)$$ be a
SRSH system whose underlying Hilbert space is $H$.
 Let $A=A^{(n)}.$  Let $X_1^{(0)}=\varnothing.$  Let
\mbox{$\phi\colon A\rightarrow K(H)$} be a non-zero 
*-homomorphism.  Then there exists an index
set $I,$ a family $(p_i)_{i\in I}$ of mutually orthogonal projections
in $B(H),$  a family $(w_i)_{i\in I}$ of isometries  in
$B(H),$ and a family $(x_i)_{i\in I}$ of elements  in
$\bigsqcup_{k=1}^n (X_k\setminus X_k^{(0)})$ (note that we do not assume that
the $x_i$ are mutually distinct) such that
\begin{enumerate}
\item $p_i\in \phi(A)'$ for all $i\in I,$ where $\phi(A)'$ denotes the commutant of
$\phi(A);$
\item $w_i^*w_i=1$ and  $w_iw_i^*=p_i$ for all $i\in I;$
\item $\phi(a)=\sum_{i\in I} w_ia(x_i)w_i^*$ for all $a\in A,$ where the
convergence is in norm;
\item $\|\phi(a)\|=\sup_{i\in i}\|a(x_i)\|$ for all $a\in A;$ 
\item $I$ is a finite set.
\end{enumerate}
\end{prop}

\begin{proof}
It is clear that $\phi(A)$ is a non-zero $\Cstar$-subalgebra of $\KO.$  
Apply Theorem \ref{RepSubAlgOfK} to $\phi(A)$ to get the index set $I$ and
the family of mutually orthogonal projections $(p_i)_{i\in I}.$  Then part 1 of the
proposition  holds 
holds.  For each $i\in I,$ define $\phi_i\colon A\rightarrow K(p_iH)$ by
$\phi_i(a)=p_i\phi(a)p_i.$  
By part 1 of this proposition, $\phi_i$ is a well defined
*-homomorphism.  It is clear that $$\phi_i(A)=p_i\phi(A)p_i\subseteq
p_iK(H)p_i=K(p_iH).$$ Then part 2 of Theorem \ref{RepSubAlgOfK} implies that 
$\phi_i(A)=K(p_iH).$  Thus $(\phi_i, p_iH)$ is an irreducible
representation of $A.$  So by Lemma \ref{IrreRepSRSHA}, there exists a
unitary $w_i\colon H\rightarrow p_iH$ and some \mbox{$x_i\in
\bigsqcup_{k=1}^n(X_k\setminus X_k^{(0)})$} such that
$\phi_i(a)=w_ia(x_i)w_i^*$ for all $a\in A.$  Identifying $w_i$ as an element
of $B(H)$ in the obvious way (identify $w_i$ 
with the composition inclusion $p_iH\rightarrow H$ followed by $w_i$), 
the element $w_i$ is an isometry in $B(H).$
Then it is clear that part 2 of this proposition  holds.  By part 4 of 
Theorem \ref{RepSubAlgOfK}, we
have $$\phi(a)=\sum_{i\in I} p_i\phi(a)p_i=\sum_{i\in
I}\phi_i(a)=\sum_{i\in I} w_ia(x_i)w_i^*$$ for all $a\in A,$ where the
convergence is in norm.  So part 3 holds.  By part 3 of Theorem
\ref{RepSubAlgOfK}, we have  $$\|\phi(a)\|=\sup_{i\in I}
\|p_i\phi(a)p_i\|=\sup_{i\in I} \|\phi_i(a)\|=\sup_{i\in
I}\|w_ia(x_i)w_i^*\|=\sup_{i\in I}\|a(x_i)\|.$$  So 4 holds. 

 Finally we
show that $I$ is a finite set by contradiction. 
Suppose that $I$ is an infinite set.  Let $S$ denote the set $\{x_i\in X\colon 
i\in I\},$ where $X=\bigsqcup_{k=1}^n X_k.$  
We claim that there are distinct $i_l\in I$ for $l\in \N$ such that
$i_l\neq i_{l'}$ if $l\neq
l',$ and that the sequence $(x_{i_l})_{l=1}^\infty$ converges to some
$x_0\in X.$  To prove this claim, if $S$ is finite, then there exists some 
$y\in S$ such that
the set $\{i\in I\colon  x_i=y\}$ is infinite.  In this case take a sequence of 
mutually distinct indices $(i_l)_{l=1}^\infty$ in $\{i\in I\colon x_i=y\}.$
Then clearly $x_{i_l}=y\rightarrow y.$  If $S$ is infinite, then, since $X$ is
compact, we can  pick a countable mutually distinct
subset elements $y_1, y_2, \ldots \in \subseteq S$ such that $y_n\rightarrow x_0$
for some $x\in X.$ For each $l\geq 1,$ choose $i_l\in I$ such that
$x_{i_l}=y_l.$ Then the indices $i_1, i_2, \ldots$ are necessarily mutually
distinct, and $x_{i_l}=y_l\rightarrow x_0.$   This proves the claim.

Now we show that for all
$a\in A,$
$\|a(x_{i_l})\|\rightarrow 0.$
Let $a\in A,$ and let $\epsilon>0.$ By part 5 of Theorem 
\ref{RepSubAlgOfK}, there exists a finite subset $F\subseteq I$ such that
$i\notin F$ implies that
$$\|p_i\phi(a)p_i\|=\|\phi_i(a)\|=\|w_ia(x_i)w_i\|=\|a(x_i)\|<\epsilon.$$
Since $F$ is finite, there exists $l_0\geq 1$ such that if $l\geq l_0$ then
$i_l\notin F.$  Thus for all $l\geq l_0,$ we have 
$\|a(x_{i_l})\|<\epsilon.$  This shows that $\|a(x_{i_l})\|\rightarrow 0$
for all $a\in A.$

Since $a$ is continuous for all $a\in A,$ we have $a(x_0)=0$ for
all $a\in A.$  Then the map $A\rightarrow \KO$ defined by $a\mapsto a(x_0)$
is the zero map, hence $x_0\in \bigsqcup_{k=1}^n X_k^{(0)},$ because  
by Lemma \ref{IrreRepSRSHA}, for all
$y\in X\setminus \left(\bigsqcup_{k=1}^n X_k^{(0)}\right),$ the map $a\mapsto a(y)$ is an
irreducible representation and hence cannot be the zero map.
Suppose that $x_0\in X_k^{(0)}$ for some $k \in \{1, \ldots , n\}.$  Now, we
assumed that the map $\phi_k\colon A^{(k-1)}\rightarrow C(X_k^{(0)}, \KO)$ is
non-vanishing, so there exists some $b\in A^{(k-1)}$ such that
$\phi_k(b)(x_0)\neq 0.$  Then, since the map $A^{(n)}\rightarrow A^{(k-1)}$ 
defined
by $(a_1, \ldots, a_n)\mapsto (a_1, \ldots, a_{k-1})$ is surjective, there
exists some $a=(a_1, \ldots, a_n)\in A$ such that $(a_1, \ldots,
a_{k-1})=b.$  Thus  $$a(x_0)=R_k(a_k)(x_0)=\phi_k(b)(x_0)\neq 0.$$
This contradicts the fact that $a(x_0)=0$ for all $a\in A.$  This means that $I$ has
to be finite.
\end{proof}

\begin{defn}
\label{SpectrumOfHom}
Let $$\left(X_1, A^{(1)}, \left(X_k, X_k^{(0)}, \psi_k, R_k, A^{(k)}\right)_{k=2}^n\right)$$ 
be a SRSH system, and let $A=A^{(n)}.$  Let 
$\phi\colon A\rightarrow \KO$ be a non-zero *-homomorphism.  Then by Proposition
\ref{CharaterOfKMapsFromSRSHA}, there exists $x_1, \ldots, x_m\in
\bigsqcup_{k=1}^n (X_k\setminus X_k^{(0)})$ and isometries $w_1, \ldots, w_m$
with orthogonal ranges such that $\phi(a)=\sum_{i=1}^m w_ia(x_i)w_i^*$ for
all $a\in A.$  We call the set $\{x_1, \ldots, x_n\}$ (not counting multiplicity)
the {\emph{spectrum of
$\phi$}}, and we will denote the spectrum of $\phi$ by $\Sp(\phi).$
Let $$\left(Y_1, B^{(1)}, \left(Y_k, Y_k^{(k)}, \phi_k, Q_k, B^{(k)}\right)_{k=2}^m\right)$$ be 
another SRSH system, let $B=B^{(m)},$ and let $\phi\colon A\rightarrow B$ be a
*-homomorphism.  We say that $\phi$ is {\emph{non-vanishing}} if, for all $y\in
\bigsqcup_{k=1}^m Y_k,$ the map $A\rightarrow \KO$ defined by ${\mathrm{ev}}_y\circ\phi$
is not the zero map.  In this case, will call $\Sp({\mathrm{ev}}_y\circ\phi)$ the {\emph{
spectrum of $\phi$ at $y$}} and write $\Sp_y(\phi).$
\end{defn}

In the previous definition, it is not necessary to insist on $\phi$ being
non-vanishing to define
$\Sp_y(\phi).$  If ${\mathrm{ev}}_y\circ \phi=0$ for some $y,$ then $\Sp_y(\phi)$
would simply be the empty set.  The condition that $\phi$ is non-vanishing
guarantees that $\Sp_y(\phi)\neq\varnothing$ for all $y\in \bigsqcup_{i=1}^m
Y_i.$

The spectrum of a *-homomorphism between homogeneous algebras was used in 
\cite{DNNPStableRankReduction} to show that simple inductive
limits of homogeneous algebras with no dimension growth have topological stable 
rank one.
One of the key steps is that if the inductive limit is simple, then the
spectra
of the connecting *-homomorphisms of the inductive system, in a sense, become more 
and more ``dense'' when we follow
the connecting maps of the inductive limit further and further out.  We will prove a
similar result in our situation.  We will first need a few preliminary results, and some 
results that will be used later in this paper.

\begin{lem}
\label{CompoIsNon-Vanishing}
Let $$\left(X_1, A^{(1)}, \left(X_k, X_k^{(0)}, \phi_k, R_k, A^{(k)}\right)_{k=2}^n\right),$$ 
$$\left(Y_1, B^{(1)}, \left(Y_k, Y_k^{(0)}, \psi_k, T_k, B^{(k)}\right)_{k=2}^m\right),$$ and 
$$\left(Z_1, C^{(1)}, \left(Z_k, Z_k^{(0)}, \theta_k, S_k, C^{(k)}\right)_{k=2}^l\right)$$ be three
SRSH systems, and let $A=A^{(n)},$ $B=B^{(m)},$ and $C=C^{(l)}.$  Let
$\phi\colon A\rightarrow B$ and $\psi\colon B\rightarrow C$ be non-vanishing
*-homomorphisms.  Then $\psi\circ \phi$ is non-vanishing.
\end{lem}

\begin{proof}
Let $z\in \bigsqcup_{i=1}^l Z_k.$  Since $\psi$ is non-vanishing, the map
${\mathrm{ev}}_z\circ \psi$ is non-zero.  So there exists $t\in \N$ with $t>0,$ and 
isometries
$w_1, \ldots, w_t,$ with orthogonal ranges
such that $\psi(b)(z)=\sum_{i=1}^t
w_ib(y_i)w_i^*$ for all $b\in B,$ where $\{y_1, \ldots,
y_t\}=\Sp_z(\psi)\neq\varnothing.$
Since $\phi$ is non-vanishing, there exists some $a\in A$ such that
$\phi(a)(y_1)\neq 0.$  Then $\|\psi(\phi(a))(z)\|\geq \|\phi(a)(y_1)\|>0,$ and hence
$\psi\circ \phi$ is non-vanishing.
\end{proof}

\begin{lem}
\label{SRSHAINon0dealOpenSet}
Let $n$ be a positive integer.  Let $$\left(X_1, A^{(1)}, \left(X_k, X_k^{(0)},
\phi_k, R_k, A^{(k)}\right)_{k=2}^n\right)$$ be a SRSH
system and let $A=A^{(n)}.$  Let $X_1^{(0)}=\varnothing$ and let
$X=\bigsqcup_{k=1}^n X_k.$  
\begin{enumerate}
\item 
Let $U\subseteq X$ be an open subset.  Then
$I_U=\{a\in
A\colon  a|_{U^c}=0\}$ is a closed two sided ideal of~$A.$  Further, let
$U_k=U\cap X_k$ for $k \in \{1, \ldots, n\},$ and let $$W_k=\left\{x\in X_k^{(0)}\colon 
\Sp_x(\phi_k)\cap \left(\bigsqcup_{i=1}^{k-1}U_i\right)\neq \varnothing\right\}$$ 
for each $k=2,
\ldots, n.$  Suppose that
\begin{align}
\label{AttachingSpectrumIntersection}
U\neq\varnothing \text{ and }
W_k=U_k\cap X_k^{(0)} \text{ for } k=2, \ldots, n.\tag*{(*)}
\end{align}
 Then $I_U\neq 0,$ and 
$$U=\{x\in X\colon  \text{there exists some }a\in I_U \text{ such that } a(x)\neq
0\}.$$
\item Let $I\subseteq A$ be a non-zero ideal. Then the set $$U=\{x\in X\colon 
\text{there exists some } a\in A \text{ such that } a(x)\neq 0\}$$ is open
in $X$ and satisfies
the condition \ref{AttachingSpectrumIntersection} in part 1. Also $I_U=I.$
\end{enumerate}
\end{lem}

\begin{proof}
For part 1, we induct on the length of the SRSH system.  If $n=1,$ then result is
trivial.  Suppose that result holds for systems of length $n,$ and let
$$\left(X_1, A^{(1)}, \left(X_k, X_k^{(0)}, \phi_k, R_k, A^{(k)}\right)_{k=2}^{n+1}\right)$$ be a
system of length $n+1.$  Let $U,$ $U_1, \ldots, U_{n+1}$ and $W_1, \ldots,
W_{n+1}$ be as given in the statement of the lemma.

It is clear that $I_U$ is a closed two sided ideal of $A.$  Let
$V=\bigsqcup_{k=1}^n U_k.$  First suppose that $V\neq \varnothing.$  Then by
the induction hypothesis, $J_V=\{a\in A^{(n)}\colon a|_{V^c}=0\}$ is a non-zero
ideal.  So let $b\in J_V$ be nonzero.  Now, for all $x\in
X_{n+1}^{(0)}\setminus W_{n+1},$ we have $\Sp_x(\phi_{n+1})\subseteq 
V^c.$  Since $b$ vanishes on $V^c,$ the function $\phi_{n+1}(b)$ also 
vanishes outside of $W_{n+1}.$  If
$W_{n+1}=\varnothing,$ then $\phi_{n+1}(b)=0.$ Thus $(b, 0)\in I_U$ and
$(b, 0)\neq 0.$  So assume that $W_{n+1}\neq\varnothing.$
Since $W_k$ is closed in $U_{n+1},$ we can extend
$\phi_{n+1}(b)$ to some $f\in C_0(U_{n+1}, \KO).$  Since $U_{n+1}\subseteq
X_{n+1}$ is
open,
we can define $f(x)=0$ for all $x\notin U_{n+1},$ so that $f\in C(X_{n+1},
\KO).$  Then $R_{n+1}(f)=\phi_{n+1}(b),$ and so $(b, f)\in I_U$ and $(b,
f)\neq 0.$  Thus $I_U\neq 0.$

Now suppose that $V=\varnothing.$  Then $W_{n+1}=\varnothing,$ and so
$U_{n+1}\subseteq X_{n+1}\setminus X_{n+1}^{(0)}.$  Since $U_{n+1}\neq\varnothing$
(otherwise $U=\varnothing$), there exists $f\in C(X_{n+1}, \KO)$ such that $f$
vanishes outside of $U_{n+1}$ and $f\neq 0.$  Then $(0, \ldots, 0, f)\in I_U$
and $(0, \ldots, 0, f)\neq 0.$  So $I_U\neq 0.$

It is clear that $$\{x\in X\colon  \text{there exists some }a\in I_U \text{ such
that } a(x)\neq 0\}\subseteq U.$$
Now let $x\in U.$  Let $k$ be the integer such that $x\in U_k.$  First
suppose that $1\leq k\leq n.$  Let $W=\bigsqcup_{i=1}^n U_i.$  Then by the
induction hypothesis, we have $$W=\{x\in X\colon  \text{there exists some
}a\in I_W\subseteq A^{(n)} \text{ such that } a(x)\neq 0\}.$$  So there
exists some $b\in I_W$ such that $b(x)\neq 0.$  An argument similar to
the one given in the second paragraph of this proof give some $f\in
C(X_{n+1}, \KO)$ such that $a=(b, f)\in I_U.$ Then $a(x)=b(x)\neq 0.$
Therefore
$$x\in \{y\in X\colon  \text{there exists some }a\in I_U \text{ such
that } a(y)\neq 0\}.$$

Now suppose that $k=n+1.$  Assume that $x\in X_{n+1}^{(0)}.$  Then $x\in
W_{n+1},$ which means that  there exists some $y\in \Sp_x(\phi_{n+1})\cap
\left(\bigsqcup_{i=1}^n U_i\right).$  By what is shown in the previous paragraph,
there exists some $a\in I_U$ such that $a(y)\neq 0.$  Then
$$\|a(x)\|=\sup_{z\in \Sp_x(\phi_{n+1})}\|a(z)\|\geq \|a(y)\|>0,$$ so
$a(x)\neq 0,$ and so $$x\in \{y\in X\colon  \text{there exists some }a\in I_U
\text{ such that } a(y)\neq 0\}.$$  Finally assume that $x\notin
X_{n+1}^{(0)}.$  Let $\xi\in \KO$ be non-zero and choose $h\in C(X_{n+1})$
such that $h(x)=1$ and  $h$ vanishes outside of $U_{n+1}\cap
(X_{n+1}\setminus X_{n+1}^{(0)}).$  Let $f=\xi h.$  Then $a=(0, \ldots, 0,
f)\in A,$ and $a$ vanishes outside of $U.$  So $a\in I_U,$ and
$a(x)=f(x)=\xi\neq 0.$  Therefore $$x\in \{y\in X\colon  \text{there exists some }a\in
I_U \text{ such that } a(y)\neq 0\}.$$  Thus $$U= \{x\in X\colon 
\text{there exists some }a\in I_U \text{ such that } a(x)\neq 0\}.$$

For part 2, we first  note that $U=\bigcup_{a\in I}\{x\in X\colon  a(x)\neq 0\}$ is open in
$X,$ and that
$U$ cannot be empty.  Let $U_1,
\ldots, U_{n+1}$ and $W_2, \ldots, W_n$ be as given in part 1.  Let $k\in \{2,
\ldots, n\}.$  Let $x\in W_k$ and let $y\in \Sp_x(\phi_k)\cap\left(\bigsqcup_{i=1}^{k-1}
U_i\right).$  Let $a\in I$ satisfy $a(y)\neq 0.$  
Then 
$$\|a(x)\|=\sup_{z\in \Sp_x(\phi_k)} \|a(z)\|\geq \|a(y)\|>0.$$  
Thus $a(x)\neq 0.$  So $x\in U_k,$ and
so $x\in U_k\cap X_k^{(0)}.$

Now suppose that $x\in U_k\cap X_k^{(0)}.$  Then $a(x)\neq 0$ for some $a\in I.$
Let $a=(b, g_1, \ldots, g_l),$ where $b\in A^{(k-1)}.$
Then $\|a(x)\|=\sup_{z\in \Sp_x(\phi_k)}\|b(z)\|.$  Now, since
$b$ vanishes outside of $\bigsqcup_{i=1}^{k-1} U_i,$ if $\Sp_x(\phi_k)\subseteq
\left(\bigsqcup_{i=1}^{k-1} U_i\right)^c,$ then $\|a(x)\|=0,$ and so $a(x)=0.$  Since
$a(x)\neq 0,$ we have $$\Sp_x(\phi_k)\cap \left(\bigsqcup_{i=1}^{k-1}
U_i\right)\neq\varnothing.$$  So $x\in W_k.$  Thus $W_k=U_k\cap X_k^{(0)}.$

It is clear that $I\subseteq I_U.$  Now we know that there exists some closed
subset $F\subseteq X$ such that $I=\{a\in A\colon a|_F=0\}.$  Since for all $x\in U,$
there exists some $a\in I$ such that $a(x)\neq 0,$ we have $F\subseteq
U^c.$  Then $a$ belonging to $I_U$ implies $a$ vanishes on $U^c,$ and so $a$ vanishes on
$F.$  So $a\in I.$  Thus $I_U\subseteq I,$ and hence $I=I_U.$
\end{proof}

\begin{lem}
\label{NonVanishingElement}
Let $$\left(X_1, A^{(1)}, \left(X_k, X_k^{(0)},
\phi_k, R_k, A^{(k)}\right)_{k=2}^n\right)$$ be a SRSH
system, and let $A=A^{(n)}.$  Let $X=\bigsqcup_{k=1}^n X_k.$ Then there exists
some $a\in A$ such that $a(x)\neq 0$ for all $x\in X.$
\end{lem}

\begin{proof}
Induct on the length of the system.  The result clearly holds for $n=1.$
Suppose that result holds for systems of length $n,$  let
$$\left(X_1, A^{(1)}, \left(X_k, X_k^{(0)}, \phi_k, R_k, A^{(k)}\right)_{k=2}^{n+1}\right)$$ be a SRSH
system, and let $A=A^{(n+1)}.$  

Now, 
$$\left(X_1, A^{(1)}, \left(X_k, X_k^{(0)}, \phi_k, R_k, A^{(k)}\right)_{k=2}^n\right)$$ is a
system of length $n,$ so by inductive hypothesis, $A^{(n)}$ contains some
$a_0$ such that $a_0(x)\neq 0$ for all $x\in \bigsqcup_{k=1}^n X_k.$  Let
$a=a_0^*a_0.$  Then $a(x)\geq 0$ for all $x\in X,$ and
$a(x)\neq 0$ for all $x\in X.$  Let $b=\phi_{n+1}(a).$  Because $a$
vanishes nowhere, and because $\phi_{n+1}$ is non-vanishing, we have 
$b(x)\neq 0$ and $b(x)\geq 0$ for all $x\in X_{n+1}^{(0)}.$  Extend $b$ to
some positive element $b'\in C(X_{n+1}, \KO).$  Let $$U=\{x\in X_{n+1}\colon 
b'(x)\neq 0\}.$$  It is clear that $U$ is an open neighborhood of $X_{n+1}^{(0)}.$  Then
$\{U, X_{n+1}\setminus X_{n+1}^{(0)}\}$ is an open cover for $X_{n+1}.$  Let
$\{h_1, h_2\}$ be a partition of unity subordinate to $\{U,
X_{n+1}\setminus X_{n+1}^{(0)}\}.$  (Without loss of generality, assume that
$\supp h_1\subseteq U,$ and $\supp h_2\subseteq X_{n+1}\setminus
X_{n+1}^{(0)}$.)
Let $\xi\in
\KO$ be a non-zero positive element.  
Let $f=h_1b'+h_2\xi.$  Then if
$x\in X_{n+1}^{(0)},$ we have $$f(x)=h_1(x)b'(x)+h_2(x)\xi
=b'(x)=b(x)=\phi_{n+1}(a)(x).$$  Thus $(a, f)\in A.$  Now let $x\in
X_{n+1}.$  If $h_1(x)\neq 0,$ then $x\in U,$ and then $h_1(x)b'(x)\neq
0.$  Since $f(x)\geq h_1(x)b'(x),$ we have $f(x)\neq 0.$  If
$h_1(x)=0,$ then $h_2(x)=1,$ and so $h_2(x)\xi=\xi\neq 0.$  Since $f(x)\geq
h_2(x)\xi,$ we have $f(x)\neq 0.$  Thus $f$ vanishes nowhere.  Then
the element $(a, f)$ vanishes nowhere on $X.$  (That is $(a, f)$ is not
contained in any non-zero proper ideal of $A.$)
\end{proof}

The next proposition shows that in a simple inductive limit in which the connecting 
maps are injective and non-vanishing, the spectra of the connecting maps
become more and more dense, in some sense.  If $A$ is a set and if $B$ is a
subset of $A$, we use $B^c$ to denote the complement of $B$.

\begin{prop}
\label{SpectrumStructureSimplIndLimSRSHA}
Let $(A_n, \psi_n)$ be an inductive system of SRSHAs and let $A$ be the
inductive limit.  Let $X_n$ be the total space for $A_n.$
Suppose that $\psi_n$ is injective for all $n,$ 
that $\psi_n$ is non-vanishing for all $n,$ and that $A$ is
simple.  Then for all $n\geq
1,$ and for all open set $U\subseteq X_n$ such that $I_U=\{a\in
A_n\colon a|_{U^c}=0\}$ is a non-zero
ideal, there exists $n_0\geq n$ such that for all $k\geq n_0$ and for all
$x\in
X_k,$ we have $\Sp_x(\psi_{n, k})\cap U\neq \varnothing,$ where $\psi_{i,
j}=\psi_{j-1}\circ\cdots \circ\psi_{i+1}\circ\psi_i$ for $i\leq j.$
\end{prop}

\begin{proof}
This will be a proof by contradiction.  Suppose that there exists $m\geq 1$
and some open set $U\subseteq X_m$ with $I_U\neq 0,$ such that  for all
$n\geq m,$ there exists some $k_n\geq n$ and some $x\in X_{k_n}$ such that
$\Sp_x(\psi_{m, k_n})\cap U=\varnothing.$  Then $U$ certainly cannot be the
entire space $X_n.$
Without loss of generality, we can 
assume that $k_n<k_{n+1}<k_{n+2}<\cdots.$  Then, passing to a subsequence
of the inductive system and truncating if necessary, we can assume that
$m=1,$ and that $k_n=n$ for all $n\geq 1.$  Thus we are assuming that there
exists some open subset $U\subseteq X_1$ with $I_U\neq 0$ such that for all
$n\geq 1,$ there exists some $x\in X_n$ such that $\Sp_x(\psi_{1, n})\cap
U=\varnothing.$  It is clear that $U\neq X_1.$ 

For each $n\geq 1,$ let $\psi^n\colon A_n\rightarrow A$ be the
natural injection that comes with the inductive limit.  
Also let 
$$V=\{x\in X_1\colon  \text{ there exists some } b\in I_U \text{ such that } b(x)\neq 0\}.$$
It is clear that $V\subseteq U.$ Then for all $n\geq 1,$ there exists some $x\in X_n$
such that $$\Sp_x(\psi_{1, n})\cap V\subseteq \Sp_x(\psi_{1, n})\cap U=\varnothing.$$
 By Lemma 
\ref{SRSHAINon0dealOpenSet}, we have $I_V=I_U\neq 0.$
 For each $n\geq 2,$
let $$F_n=\overline{\{x\in X_n\colon \Sp_x(\psi_{1, n})\cap V=\varnothing\}}.$$
Then
$F_n\neq \varnothing$ for all $n\geq 2.$
Let $I_n=\{a\in A_n\colon a|_{F_n}=0\}.$ Let $I_1=I_V.$ 
For each $n\geq 1,$ let
$J_n=\psi^n(I_n),$ and let $B_n=\psi^n(A_n).$  Then $J_n$ is a closed
two sided ideal of $B_n.$  We first show that $J_1\subseteq J_2\subseteq
J_3\subseteq\cdots.$  Fix $n\geq 1,$ and let $a\in I_n.$   Let $x_0\in
\{x\in X_{n+1}\colon  \Sp_x(\psi_{1, n+1})\cap V=\varnothing\}.$  
Let $y\in \Sp_{x_0}(\psi_n).$

Suppose that $\Sp_y(\psi_{1, n})\cap V\neq \varnothing.$  
Let $z\in \Sp_y(\psi_{1, n})\cap V,$ and let
$b\in I_1=I_V$ be such that $b(z)\neq 0.$  
Then $$\|\psi_{1,
n+1}(b)(x_0)\|=\|\psi_n(\psi_{1, n}(b))(x_0)\|\geq \|\psi_{1,
n}(b)(y)\|\geq \|b(z)\|>0.$$  But $b$ vanishes outside of $V,$ so if $x\in
X_{n+1}$ satisfies $\Sp_x(\psi_{1, n+1})\cap V=\varnothing,$ then 

$$\|\psi_{1, n+1}(b)(x)\|=\sup_{z'\in \Sp_x(\psi_{1, n+1})}\|b(z')\|=0;$$
hence
in particular $\psi_{1, n+1}(b)(x_0)=0.$  
 This contradicts the fact that
$\|\psi_{1, n+1}(b)(x_0)\|>0.$  Thus $\Sp_y(\psi_{1, n})\cap V=\varnothing.$

Then $y\in F_n,$ and so $a(y)=0.$  This holds for all $y\in
\Sp_{x_0}(\psi_n),$ so $\psi_n(a)(x_0)=0.$  This holds for all $x_0\in
X_{n+1}$ such that $\Sp_{x_0}(\psi_{1, n+1})\cap U=\varnothing,$ so
$\psi_n(a)|_{F_{n+1}}=0,$ and so $\psi_n(a)\in I_{n+1}.$  Then
$\psi^n(a)=\psi^{n+1}(\psi_n(a))\in \psi^{n+1}(I_{n+1})=J_{n+1}.$  This
holds for all $a\in I_n,$ so $J_n=\psi^n(I_n)\subseteq J_{n+1}.$
This holds for all $n\geq 1,$ so we have $J_1\subseteq J_2\subseteq
\cdots.$ 

Then $J=\overline{\bigcup_{n\geq 1}J_n}$ is an ideal of $A.$  The ideal 
$J$ cannot be 0,
because $\psi^1$ is injective and $I_1\neq 0.$  Finally we show that
$J\neq A.$  Let $a\in A_1$ satisfy $a(x)\neq 0$ for all $x\in X_1.$  Then
compactness of $X_1$ gives that there exists $\epsilon>0$ such that
$\|a(x)\|\geq \epsilon$ for all $x\in X_1.$  For all $n\geq 2$ and for all
$x\in X_n,$ we have $\|\psi_{1, n}(a)(x)\|=\sup_{y\in \Sp_x(\psi_{1,
n})}\|a(y)\|\geq \epsilon.$  For all $n\geq 2,$ and  for
all $b\in I_n,$ we have $$\|\psi_{1, n}(a)-b\|\geq \|\psi_{1,
n}(a)|_{F_n}-b|_{F_n}\|=\|\psi_{1, n}(a)|_{F_n}\|\geq \epsilon.$$  Then
for all $n\geq 1$ and for all $b\in I_n,$ we have
$$\|\psi^1(a)-\psi^n(b)\|=\|\psi^n(\psi_{1, n}(a))-\psi^n(b)\|=\|\psi_{1,
n}(a)-b\|\geq \epsilon.$$  Thus $\psi^1(a)\notin J.$  So $J\neq A.$  

This shows that $J$ is a non-zero proper ideal of $A,$ which
contradicts the simplicity of $A.$

\end{proof}

\section{Topological Stable Rank of Simple Inductive Limits of SRSHAs}

We begin this section by writing down some results about semi-continuity 
of spectral projections at 
self-adjoint elements in $\KO,$ which we will use later on. 
Then, through several lemmas, we adapt Lemma 3.3 in 
\cite{PhillipsStableRankRSHA}, which is the key lemma
in showing that simple inductive limits of RSHAs with no dimension growth have
topological stable rank one, to our situation. The last portion of the section will
be dedicated to showing that if $A$ is simple inductive limit of SRSHAs
with no dimension growth such that all the connecting maps are injective 
and non-vanishing, then $A$ has topological stable rank one.

\begin{lem}
\label{ShiftedFuncCalc}
Let $A$ be a $\Cstar$-algebra,  let $\widetilde A$ denote the unitization of
$A,$ and let $1$ be the adjoined identity.  (Here, we add a new identity to 
$A$ even if $A$ is already unital.) Let $a\in A$ be self-adjoint and
let
$\widetilde a=a+1.$ Then
\begin{enumerate}
\item $\Sp(a)+1=\Sp(\widetilde a)$ where both spectra are taken with respect to
$\widetilde A.$
\item Let $h\colon \Sp(\widetilde a )\rightarrow \Sp(a)$ be defined by $h(\xi)=\xi-1$
and let $h^*\colon C(\Sp(a))\rightarrow C(\Sp(\widetilde a))$ be defined by
$h^*(f)=f\circ h.$  Let $F\colon C(\Sp(a))\rightarrow \widetilde A$ and let $\widetilde
F\colon C(\Sp(\widetilde a))\rightarrow \widetilde A$ be the functional calculus (with
respect to $\widetilde A$) at $a$
and $\widetilde a$ respectively.  Then $F=\widetilde F\circ h^*.$
\end{enumerate}
\end{lem}

\begin{proof}
Part 1 is trivial. To prove part 2, note that $\widetilde a = h^{-1}(a)$.  Then
if $f \in C( \Sp (a)),$ we have 
\begin{align*}
\widetilde F \circ h^* (f) 
&= h^*(f) (\widetilde a) = h^*(f) ( h^{-1} (a) )\\
&= (f \circ h) ( h^{-1} (a) ) = (f \circ h \circ h^{-1}) (a) = f(a)
= F(f).
\end{align*} 
\end{proof}

 For all $\Cstar$-algebras $A$ and all $a\in A,$ we use
$|a|$ to denote $(a^*a)^{1/2}.$   
We use
$\chi_\alpha\colon \R \rightarrow \R$ to denote the characteristic function of 
$(-\infty, \alpha)$ for all $\alpha\in \R.$
Also, for all $\Cstar$-algebras $A$ and all  self-adjoint $a\in A$, we use
$p_\alpha(a)$ to denote $\chi_\alpha(a).$  Even though $p_\alpha(a)$ may
not be in $A$ for some combinations of $a,$ $A$ and $\alpha,$ it is still in
the double commutant of $A$ when $A$ is faithfully represented on a 
Hilbert space.  For our purposes, $A$ will be either the
algebras of compact
operators on separable Hilbert spaces, or their unitization; 
and $\alpha$ will be less then the limit point of $\Sp(a)$ (if any).  In
these cases $p_\alpha(a)$ will  be a finite rank projection, and hence in $A.$
Then the next corollary follows immediately from Lemma \ref{ShiftedFuncCalc}.

\begin{cor}
\label{ShiftedSpProj}
Let $a\in \KO_{s.a.},$ let $1>\alpha>0,$ and let $\widetilde a=a+1.$  Then
$p_\alpha(\widetilde a)=p_{\alpha-1}(a).$
\end{cor}

\begin{lem}
\label{FunCalcOfTwoDirectSum}
Let $A$ be a unital $\Cstar$-algebra and let $p_1, p_2\in A$ be orthogonal
projections such that $p_1+p_2=1.$  Let $A_1$ and $A_2$ be
$\Cstar$-subalgebras of $A$ such that $p_i$ is the identity of $A_i$ for
$i=1, 2.$ Let $a_1\in A_1$ and $a_2\in A_2.$  
\begin{enumerate}
\item Then $\Sp_A(a_1+a_2)=\Sp_{A_1}(a_1)\cup \Sp_{A_2}(a_2),$ where $\Sp_B(b)$
denotes the spectrum of $b$ with respect to $B$ for all $\Cstar$-algebra
$B$ and any $b\in B.$
\item Suppose that $a_1$ and $a_2$ are self-adjoint.  
Let $F_i$ be the functional
calculus of $a_i$ with respect to $A_i,$ for $i=1, 2,$ and let $F$ be the
functional calculus of $a_1+a_2$ with respect to $A.$ 
Then for all $f\in C(\Sp_A(a_1+a_2)),$ we have $F(f)=F_1(f)+F_2(f),$ that is,
$f(a_1+a_2)=f(a_1)+f(a_2).$
\end{enumerate}
\end{lem}

\begin{proof}
First assume that $A_i=p_iAp_i$ for $i=1, 2.$  Let $\lambda\in \K.$
If $\lambda-(a_1+a_2)$ is invertible in $A,$ then there exists some $b\in
A$ such that $b(\lambda-a_1-a_2)=(\lambda-a_1-a_2)b=1=p_1+p_2,$ and
$b$ commutes with $p_1$ and $p_2.$ So  
$p_1bp_1$ and $p_2bp_2$ are the inverses of $\lambda p_1-a_1$ and $\lambda
p_2-a_2$ in $A_1$ and $A_2,$ respectively, and so $\lambda p_1-a_1$ and
$\lambda p_2-a_2$ are both invertible.  On the other hand, if both
$\lambda p_1-a_1$ and $\lambda p_2-a_2$ are invertible, then there exists
$b_i\in A_i$ such that $b_i=(\lambda p_i-a_i)^{-1}$ for $i=1, 2.$  Then
$b_1+b_2=(\lambda-a_1-a_2)^{-1}.$  Thus $\lambda\notin \Sp_A(a_1+a_2)$ if
and only if $\lambda\notin \Sp_{A_1}(a_1)\cap \Sp_{A_2}(a_2).$  So result
follows.  
Now assume that $A_i$ is an arbitrary $\Cstar$-algebra of $A$ that contains
$p_i$ as its identity, for $i=1, 2.$  Then for $i=1, 2,$ $A_i$ is a
$\Cstar$-algebra of $p_iAp_i$ that contains the identity of $p_iAp_i,$ so
$\Sp_{p_iAp_i}(a_i)=\Sp_{A_i}(a_i).$  Thus  
$$\Sp_A(a_1+a_2)=\Sp_{p_1Ap_1}(a_1)\cup\Sp_{p_2Ap_2}(a_2)
=\Sp_{A_1}(a_1)\cup\Sp_{A_2}(a_2),$$ and part 1 or the lemma is proven.

Since $a_1a_2=a_2a_1=0,$ it is easy to verify that if $\pi$ is a
polynomial on $\Sp_{A}(a_1+a_2),$ then $\pi(a_1)+\pi(a_2)=\pi(a_1+a_2),$
where functional calculus on the left side of the equation is taken in the
subalgebras $A_i,$ $i=1, 2,$ and the functional calculus on the right side
of the equation is taken in $A.$  So the continuous map
$C(\Sp_A(a_1+a_2))\rightarrow A$ defined by $f\mapsto f(a_1)+f(a_2),$ where
the respective functional calculus is taken in the subalgebra, agrees with
the map $f\mapsto f(a_1+a_2)$  on the
set of all polynomials, which is dense in $C(\Sp_A(a_1+a_2)).$
Hence the result follows.
\end{proof}

From \ref{FunCalcOfTwoDirectSum}, a standard induction argument shows the
following:
\begin{cor}
\label{FunCalcOfMultipleDirectSum}
Let $A$ be a unital $\Cstar$-algebra, and let $p_1,\ldots, p_n\in A$ be orthogonal
projections such that $p_1+p_2+\cdots +p_n=1.$  Let $A_i$ be a
$\Cstar$-subalgebra of $A$ such that $p_i$ is the identity of $A_i$ for
$i=1, 2, \ldots, n.$ Let $a_i\in A_i,$ $i \in \{1, \ldots, n\}.$
\begin{enumerate}
\item Then $\Sp_A\left(\sum_{i=1}^nf
a_i\right)=\bigcup_{i=1}^n\Sp_{A_i}(a_i).$
\item Suppose that $a_i$ is self-adjoint for $i \in \{1, \ldots, n\}.$
Let $F_i$ be the functional
calculus of $a_i$ with respect to $A_i$ for $i \in \{1, \ldots, n\}$ 
and let $F$ be the
functional calculus of $\sum_{i=1}^na_i$ with respect to $A.$ 
Then for all $f\in C\left(\Sp_A\left(\sum_{i=1}^n
a_i\right)\right),$ 
we have $F(f)=\sum_{i=1}^n F_i(f),$ that is,
$f\left(\sum_{i=1}^n a_i\right)=\sum_{i=1}^n
f(a_i).$
\end{enumerate}
\end{cor}

The next few results are about the semicontinuity of spectral projections.

\begin{lem}
\label{SemiContinuityOfSpectrumProjOfCptOp}
Let $\epsilon>0,$ let $0<\alpha_1<\alpha_2<1,$ and let $M\geq 1$ be a
real
number.  Then there exists some
$\delta>0$ such that if $a, b\in \KO_{s.a.},$ $\widetilde a=a+1,$ 
$\widetilde b=b+1,$ $\|\widetilde a\|\leq M,$ $\|\widetilde b\|\leq M,$ 
and $\|\widetilde a-\widetilde b\|<\delta,$ then 
$$\|p_{\alpha_1}(\widetilde a)p_{\alpha_2}(\widetilde
b)- p_{\alpha_1}(\widetilde a)\|<\epsilon$$
and 
$$\rank(p_{\alpha_1}(\widetilde a))\leq \rank(p_{\alpha_2}(\widetilde b)).$$
\end{lem}

\begin{proof}
We know that there exists a $\sigma_0>0$ such that 
if $p, q$ are projections in $\KO$ 
such that
$\|pq-q\|<\sigma_0,$ then $\rank(q)\leq \rank(p).$ Let
$\sigma=\min\{\epsilon, \sigma_0\}.$

Define $f\colon [-M, M]\rightarrow [0, 1]$ by 
$$f(t)=\begin{cases}
1&t\in [-M, \alpha_1]\\
\frac{\alpha_2-t}{\alpha_2-\alpha_1}&
t\in [\alpha_1, \alpha_2]\\
0&t\in [\alpha_2, M].
\end{cases}$$
Then it is clear that $f\in C([-M, M]).$  
Use functional calculus to obtain a positive real number $\delta$ such that if $A$ is any unital
$\Cstar$-algebra, and if $a, b \in A$ are self-adjoint elements with $\|a\| \leq M,$ $\|b\|
\leq M,$ and $\| a-b\| < \delta,$ then $\| f(a) - f(b) \| < \sigma/2.$
Let $a, b\in \KO_{s.a.},$  $\widetilde a=a+1,$ and
$\widetilde b=b+1.$  Then $\widetilde a, \widetilde b\in \widetilde \KO,$ which is unital.
Suppose that $\|\widetilde a\|\leq M,$ $\|\widetilde b\|\leq M,$ and that $\|\widetilde
a-\widetilde b\|<\delta.$  By the choice of $\delta,$ we have $\|f(\widetilde
a)-f(\widetilde b)\|<\sigma/2.$  Now, $\chi_{\alpha_1}f=\chi_{\alpha_1}$ 
and $\chi_{\alpha_2}f=f$ on
$[-M, M].$  Thus $p_{\alpha_1}(\widetilde a)f(\widetilde
a)=p_{\alpha_1}(\widetilde a),$ and $p_{\alpha_2}(\widetilde b)f(\widetilde
b)=f(\widetilde b).$
Then we have
\begin{align*}
\|p_{\alpha_1}(\widetilde a)-p_{\alpha_1}(\widetilde a)p_{\alpha_2}(\widetilde b)\|
&=\|p_{\alpha_1}(\widetilde a)f(\widetilde a)-p_{\alpha_1}(\widetilde a)f(\widetilde
a)p_{\alpha_2}(\widetilde
b)\|\\
&\leq\|p_{\alpha_1}(\widetilde a)f(\widetilde a)-p_{\alpha_1}(\widetilde a)f(\widetilde b)\|\\
&\hspace*{3em}\mbox{}+\|p_{\alpha_1}(\widetilde a)f(\widetilde b)-p_{\alpha_1}(\widetilde a)f(\widetilde
a)p_{\alpha_2}(\widetilde
b)\|\\
&\leq\|f(\widetilde a)-f(\widetilde b)\|
+\|f(\widetilde b)-f(\widetilde a)p_{\alpha_2}(\widetilde b)\|\\
&= \|f(\widetilde a)-f(\widetilde b)\|
+\|f(\widetilde b)p_{\alpha_2}(\widetilde b)-f(\widetilde a)p_{\alpha_2}(\widetilde b)\|\\
&\leq \|f(\widetilde a)-f(\widetilde b)\| +\|f(\widetilde b)-f(\widetilde a)\|\\
&< \sigma\leq \epsilon.
\end{align*}
Then by the choice of $\sigma,$ we have 
$\rank(p_{\alpha_1}(\widetilde a))\leq \rank(p_{\alpha_2}(\widetilde b)).$
\end{proof}

\begin{cor}
\label{SemiContinuityOfSpectrumProjOfC(X,K)}
Let $\epsilon>0,$ let $0\leq\alpha_1<\alpha_2<1,$ and let $M\geq 1$ be a
real number.  Then there exists $\delta>0$ such that
if $X$ is compact Hausdorff space, and if $a, b\in C(X, \KO)_{s.a.},$ $\widetilde
a=a+1,$  $\widetilde b=b+1,$ $\|\widetilde a\|\leq M,$ $\|\widetilde b\|\leq M,$ and
$\|\widetilde a-\widetilde b\|<\delta,$  then 
$$\|p_{\alpha_1}(\widetilde a(x))p_{\alpha_2}(\widetilde b(x))-p_{\alpha_1}(\widetilde
a(x))\|<\epsilon,\;\;\;\;
\text{ for all }x\in X;$$
and $$\rank(p_{\alpha_1}(\widetilde a(x)))
\leq \rank(p_{\alpha_2}(\widetilde b(x))),\;\;\;\;
\text{ for all }x\in X..$$
\end{cor}

\begin{proof}
First of all, we identify $\widetilde{C(X, \KO)}$ as a subalgebra of $C(X,
\widetilde{\KO})$ by identifying $(a, \lambda)\in \widetilde{C(X, \KO)}$ with
$a+\lambda 1_X,$
where $1_X$ is the constant function on $X$ at $id_H.$  Then it is clear
that $\widetilde a(x)=\widetilde{a(x)}$ for all $x\in X.$

Apply \ref{SemiContinuityOfSpectrumProjOfCptOp} to $\epsilon, \alpha_1,
\alpha_2$ and $M$ to get a $\delta>0.$  The result follows.
\end{proof}

\begin{cor}
\label{RankBdnessSpProj}
Let $X$ be a compact Hausdorff space, let $0<\alpha<1,$ let $a\in C(X,
\KO)_{s.a.},$ let $\widetilde a=a+1.$  Then there exists some $n\in \N$ such
that $\rank(p_{\alpha}(\widetilde a(x)))\leq n$ for all $x\in
X.$
\end{cor}

\begin{proof}
If $a=0,$ then nothing to prove.  So assume $a\neq 0.$

Let $\alpha<\sigma<1.$  Apply Corollary
\ref{SemiContinuityOfSpectrumProjOfCptOp} to $\epsilon=1,$
$0<\alpha<\sigma<1,$ and $M=\|\widetilde a\|,$ to get $\delta>0.$  For each $x\in
X,$ let $U_x=\{y\in X\colon  \|\widetilde a(x)-\widetilde a(y)\|<\delta\}.$  Then
there
exists $x_1, \ldots, x_m\in X$ such that $\bigcup_{i=1}^m U_{x_i}=X.$  Let
$n=\max\{\rank(p_{\sigma}(\widetilde a(x_i)))\colon i=1, \ldots, m\}.$
Let $x\in X.$  Then $x\in U_{x_k}$ for some $k.$  So $\|\widetilde a(x)-\widetilde
a(x_k)\|<\delta.$  Also $\|\widetilde a(x)\|\leq \|\widetilde a\|$ and $\|\widetilde
a(x_k)\|\leq \|\widetilde a\|.$  So by the choice of $\delta,$ we have 
$\rank(p_{\alpha}(\widetilde a(x)))\leq \rank(p_{\sigma}(\widetilde a(x_k)))\leq n.$
\end{proof}

\begin{lem}
\label{HorizontalCtractSpProj}
Let $n\geq \N,$ let $\alpha>0,$ let $M>0$ be a real number, and let $a\in
M_n$ be self-adjoint.  Then
$p_{\alpha}(a)=p_{\alpha/M}(a/M).$
\end{lem}

\begin{proof}
Let $\Sp(a)\cap (-\infty, \alpha)=\{r_1, \ldots,
r_k\}.$  Then $$\Sp(a/M)\cap (-\infty, \alpha/M)
=\{r_1/M, r_2/M, \ldots,
r_k/M\}.$$  Then $p_{\alpha}(a)=\sum_{i=1}^k p_i,$ where $p_i$
is the projection to the eigenspace of $a$ corresponding to $r_i,$ and
$p_{\alpha/M}(a/M)=\sum_{i=1}^k q_i,$ where $q_i$ is the
projection onto the eigenspace of $a/M$ corresponding to $r_i/M.$  But for
all $i \in \{1, \ldots, k\}$ and 
all $\xi\in \K^n,$ $a(\xi)=r_i\xi$ if and only if $(a/M)(\xi)=(r_i/M)\xi.$
So $p_i=q_i$ for all $i \in \{1, \ldots, n\},$ and so the result follows.
\end{proof}

\begin{lem}
\label{SemiContOfSpProj2}
Let $1>\alpha>0,$ let $a\in \KO_{s.a.},$ and let $\widetilde a=a+1\in
\widetilde{\KO}.$  Then there exists a $\delta>0$ such that if $b\in
\widetilde{\KO}_{s.a.},$ and if $\|b-\widetilde a\|<\delta,$ then 
$\rank(p_\alpha(\widetilde a))\leq \rank(p_\alpha(b)).$
\end{lem}

\begin{proof}
Fix $1>\alpha>0$ and $a\in \KO_{s.a.}.$  Since $\alpha<1,$ $\Sp(\widetilde a)\cap
(-\infty, \alpha)$ is a finite set.  So there exists $\delta_1>0$ such that
$\Sp(\widetilde a)\cap (\alpha-3\delta_1, \alpha+3\delta_1)\subseteq
\{\alpha\}.$  Let $F_1=[-\|\widetilde a\|-\delta_1, \alpha-2\delta_1],$ and
$F_2=[\alpha-\delta_1, \|\widetilde a\|+\delta_1].$  Then $$\Sp(\widetilde
a)\subseteq (-\|\widetilde a\|-\delta_1, \alpha-2\delta_1)
\cup(\alpha-\delta_1, \|\widetilde a\|+\delta_1)\subseteq F_1\cup F_2.$$  Let
$K=F_1\cup F_2.$  Let $\phi= \chi_{F_1}.$  Then $\phi\in C(K).$  Since
$K\subseteq \R$ is compact, there exists a polynomial $\pi\in C(K)$ such
that
$\|\pi-\phi\|_\infty<1/3.$  The map $x\mapsto \pi(x)$ is continuous, so
there exists $\delta_2>0$ such that if $\|x-\widetilde a\|<\delta_2,$ then
$\|\pi(x)-\pi(\widetilde a)\|<1/4.$  Let $\delta=\min\{\delta_1/2, \delta_2\}.$  

Let $b\in \widetilde{\KO}_{s.a.}$ satisfy $\|b-\widetilde a\|<\delta.$  Then
$\Sp(b)\subseteq \cup\{(r-\delta, r+\delta)\colon r\in \Sp(\widetilde a)\}.$  If
$r\in \Sp(\widetilde a),$ then $-\|\widetilde a\|\leq r\leq \alpha-3\delta_1$ or
$\alpha\leq r\leq \|\widetilde a\|,$ and then $$(r-\delta, r+\delta)\subseteq 
(-\|\widetilde a\|-\delta, \alpha-3\delta_1+\delta)\cup (\alpha-\delta, \|\widetilde
a\|+\delta).$$  So 
\begin{align*}
\Sp(b)&\subseteq(-\|\widetilde a\|-\delta,
\alpha-3\delta_1+\delta)\cup (\alpha-\delta, \|\widetilde a\|+\delta)\\
&\subseteq (-\|\widetilde a\|-\delta_1, \alpha-2\delta_1)\cup (\alpha-\delta_1,
\|\widetilde a\|+\delta_1)\subseteq K.
\end{align*}
Then $$\|\phi(\widetilde a)-\phi(b)\|
\leq \|\phi(\widetilde a)-\pi(\widetilde a)\|+\|\pi(\widetilde
a)-\pi(b)\|+\|\pi(b)-\phi(b)\|<1.$$  Thus $\phi(\widetilde a)$ and 
$\phi(b)$ are unitarily equivalent projections, and so $\rank(\phi(\widetilde
a))=\rank(\phi(b)).$  But $\phi(\widetilde a)=p_\alpha(\widetilde a),$ so
$\rank(p_\alpha(\widetilde a))=\rank(\phi(b)).$  Also $\phi\leq \chi_{(-\infty,
\alpha)},$ so $\phi(b)\leq p_\alpha(b),$ and so $\rank(p_\alpha(\widetilde
a))= \rank(\phi(b))\leq p_\alpha(b).$
\end{proof}

The remaining portion of this section will be dedicated to obtaining a 
topological stable rank reduction theorem for SRSHAs.  The idea is to obtain an 
approximate polar decomposition for elements $a$ in a SRSHA
such that the dimensions of the eigenspaces of $|a(x)|$ corresponding to 
small eigenvalues are large enough for every $x\in X.$  
This can be easily done in $\widetilde{C(X, \KO)},$ where $X$ is just 
a one-point space and $\widetilde{C(X, \KO)}$ denotes the unitization of
$C(X, \KO),$ which can always
be taken to be the first base space of any SRSH system.  We then 
have an approximate polar decomposition for the image of the first
coordinate of $a$ under
the first attaching map.  In order to obtain an approximate polar
decomposition for $a,$ we will need to be able to extend the image of 
the unitary used in the approximate polar decomposition for the first 
coordinate of the element $a$ to a unitary in $\widetilde{C(X_2, \KO)},$ 
where $X_2$ is the second
base space in the SRSH system.  Thus we will need an extension result for 
such unitaries.  This extension result for RSHAs is given by Lemma 3.3 in 
\cite{PhillipsStableRankRSHA}.  We will modify this lemma to suit our 
situation.

The following lemma is a slight modification of Lemma 3.3 in \cite{PhillipsStableRankRSHA}.  
In fact, the original proof of Lemma 3.3 in \cite{PhillipsStableRankRSHA} 
also proves the following lemma.

\begin{lem}
\label{Lem3.3PhillipsRSHA2MatrixSizeRemoved}
Let $\epsilon, \alpha>0$ and let $n\in \N.$  Then there exists a
$\delta>0$ such that the following holds.  Let $X$ be a compact Hausdorff
space with $\dim(X)=d<\infty,$ and let $X^{(0)}\subseteq X$ be a closed
subspace.  Let $m\in \N,$ and let $a\in C(X, M_m)$ satisfy $\|a\|\leq
1.$  For each $x\in X,$ let 
$$p(x)=\chi_{(-\infty, \alpha)}([a(x)^*a(x)]^{1/2}).$$  
Suppose that $n\geq \rank(p(x))\geq d/2$ for all $x\in X.$ Let $u^{(0)}\in
U_0(C(X, M_m))$ be a unitary such that 
$$\|[u^{(0)}(x)[a(x)^*a(x)]^{1/2}-a(x)][1-p(x)]\|<\delta$$
for every $x\in X^{(0)}.$  Let $t\mapsto u_t^{(0)}$ be a homotopy from 1 to
$u^{(0)}$ in $U(C(X^{(0)}, M_m)).$ Then there exists a unitary $u\in U_0(C(X,
M_m))$ and a homotopy $t\rightarrow u_t$ in $U(C(X, M_m))$ from $1$ to $u$
such that
$u|_{X^{(0)}}=u^{(0)},$ $u_t|_{X^{(0)}}=u_t^{(0)}$ for all $t,$ and such that
$$\|[u(x)[a(x)^*a(x)]^{1/2}-a(x)][1-p(x)]\|<\epsilon$$
for all $x\in X.$  
\end{lem}

Now we remove the condition that the element $\|a\|$ has norm less or
equal to 1 from Lemma \ref{Lem3.3PhillipsRSHA2MatrixSizeRemoved}.

\begin{cor}
\label{Lem3.3PhillipsRSHA2Norm1MatrixSizeRemoved}
Let $\epsilon, \alpha>0,$ let $n\in \N,$ and let $M\geq 1$ be a real number.  
Then there exists a
$\delta>0$ such that the following holds.  Let $X$ be a compact Hausdorff
space with $\dim(X)=d<\infty,$ and let $X^{(0)}\subseteq X$ be a closed
subspace.  Let $m\in \N,$ and let $a\in C(X, M_m)$ satisfy $\|a\|\leq
M.$  For each $x\in X,$ let 
$$p(x)=p_{\alpha}(|a(x)|).$$  
Suppose that $n\geq \rank(p(x))\geq d/2$ for all $x\in X.$ Let $u^{(0)}\in
U_0(C(X^{(0)}, M_m))$ be a unitary such that 
$$\|[u^{(0)}(x)|a(x)|-a(x)][1-p(x)]\|<\delta$$
for every $x\in X^{(0)}.$ Let $t\mapsto u_t^{(0)}$ be a homotopy in
$U(C(X^{(0)}, M_m))$ from 1 to $u^{(0)}.$   Then there exists a unitary
$u\in U_0(C(X, M_m))$ and a homotopy $t\mapsto u_t$ in $U(C(X, M_m))$ from
1 to $u$ such that $u|_{X^{(0)}}=u^{(0)},$ $u_t|_{X^{(0)}}=u_t^{(0)}$ for
all $t,$ and that
$$\|[u(x)|a(x)|-a(x)][1-p(x)]\|<\epsilon$$
for all $x\in X.$
\end{cor}

\begin{proof}
Apply Lemma \ref{Lem3.3PhillipsRSHA2MatrixSizeRemoved} to $\epsilon/M,
\alpha/M,$ $n$ to get $\delta.$  Let $X,$ $X^{(0)},$ $m,$ $a,$ $p,$ $u^{(0)}$
be as given in the statement of this corollary.  Let $t\mapsto u_t^{(0)}$
be a path from 1 to $u^{(0)}.$

Let $b=a/M.$ Then $\|b\|\leq 1.$   
Let $q(x)=p_{\alpha/M}(|b(x)|).$
By Lemma \ref{HorizontalCtractSpProj}, we have
\mbox{$q(x)=p(x)$} for all $x\in X.$  Then we have 
$n\geq \rank(q(x))\geq d/2$ for all $x\in X.$ Also, 
$$\|[u^{(0)}(x)|b(x)|-b(x)][1-q(x)]\|
<\delta/M\leq \delta$$ for all $x\in X^{(0)}.$  So by the choice of
$\delta,$ there exists a unitary $u\in U_0(C(X, M_m)),$ and a homotopy
$t\mapsto u_t$ in $U(C(X, M_m))$ from 1 to $u$ such that
$u|_{X^{(0)}}=u^{(0)},$ $u_t|_{X^{(0)}}$ for all $t,$ and that 
$$\|[u(x)|b(x)|-b(x)][1-q(x)]\|<\epsilon/M.$$
Then $$\|[u(x)|a(x)|-a(x)][1-p(x)]\|<M\cdot \frac{\epsilon}{M}=\epsilon.$$
\end{proof}

The next lemma adapts the above to unitizations of $C(X)\otimes M_n.$

\begin{lem}
\label{UnitizedLem3.3PhillpsRSHA2}
Let $1>\alpha, \epsilon>0,$ let $n\in \N,$ and let $M\in [1, \infty).$  Then
there exists $\delta>0$ such that the following holds.
Let $X$ be a compact Hausdorff space such that $\dim(X)=d<\infty,$ and let
$Y$ be a closed subspace.  Let $m\in \N,$ let $a\in C(X, M_m),$
and let $\widetilde a=a+1_X\in {C(X, M_m)^\sim},$ where $1_X$ denotes the adjoined
identity.   Suppose that $\|\widetilde a\|\leq M.$  For each $x\in X,$ let 
$\widetilde p(x)=p_{\alpha}(|\widetilde a(x)|).$
Suppose that $n\geq \rank(\widetilde p(x))\geq d/2.$  Let $u_0\in
U_0({C(Y, M_m)^\sim})$ satisfy 
\begin{equation}
\label{Eq3.[UnitizedLem3.3PhillpsRSHA2].1}
\left\|[u_0(x)|\widetilde a(x)|-\widetilde
a(x)][1-\widetilde p(x)]\right\|<\delta\;\;\;\;\text{for all }x\in Y.
\end{equation}
Let $t\mapsto w_t$ be a homotopy in $U(C(Y, M_m)^\sim)$ from 1 to $u_0.$
Then there exists a unitary $u$ contained in $U_0({C(X, M_m)^\sim})$ and a homotopy
$t\rightarrow v_t$ in $U(C(X, M_m)^\sim)$ from 1 to $u$ such that
$u|_Y=u_0,$ $v_t|_Y=w_t$ for all $t,$ and that

\begin{equation}
\label{Eq3.[UnitizedLem3.3PhillpsRSHA2].2}
\left\|[u(x)|\widetilde a(x)|-\widetilde
a(x)][1-\widetilde p(x)]\right\|<\delta\;\;\;\;\text{for all }x\in X,
\end{equation}
\end{lem}

\begin{proof}
Let $0<\epsilon, \alpha<1,$ $n\in \N,$ and $M\in [1, \infty)$ be given.
Apply Corollary \ref{Lem3.3PhillipsRSHA2Norm1MatrixSizeRemoved} to 
$\epsilon,$ $\alpha,$ $n,$ and $M$ to obtain $\delta'>0,$ and let
$\delta=\min\{\epsilon, \delta'/2\}.$  Let $X,$ $Y,$ $m,$ $a,$
$\widetilde p,$ and $u_0$ satisfy the conditions in the statement of the
lemma.  Let $t\mapsto w_t$ be a homotopy in $U(C(Y, M_m)^\sim)$ from $1$ to
$u_0.$

We set up some notations first.  
We use 1 to denote the adjoined identity of $\widetilde{M_m},$ and use
$e$ to
denote the identity of $M_m.$ Use $1_X$ and $1_Y$ to denote
the adjoined identity of $C(X, M_m)^\sim$ and $C(Y, M_m)^\sim,$
respectively.  Use $e_X$ and $e_Y$ to denote the identities of $C(X, M_m)$
and $C(Y, M_m)$ respectively.

For each $x\in X,$ or $Y,$ use
${\mathrm{ev}}_x$ to denote the map $C(X, M_m)\rightarrow M_m,$ or 
$C(Y, M_m)\rightarrow M_m,$ defined by ${\mathrm{ev}}_x(a)=a(x).$
By identifying $(a, \lambda)$ with $a+\lambda\cdot 1_X,$ or $a+\lambda\cdot
1_Y,$ we treat 
$C(X, M_m)^\sim$ and $C(Y, M_m)^\sim$ as subalgebras of 
$C(X, \widetilde{M_m})$ and $C(Y, \widetilde{M_m})$ respectively.  For each
$x\in X,$ or $Y,$ use $\widetilde{{\mathrm{ev}}_x}$ to denote the map
$C(X, M_m)^\sim\rightarrow \widetilde{M_m}$ or
$C(Y, M_m)^\sim\rightarrow \widetilde{M_m},$ defined by
$\widetilde{{\mathrm{ev}}_x}(a)=a(x).$
Let $\tau$ denote the standard map from the unitization of any
$\Cstar$-algebra to $\K.$

Define 
$$\Phi_X\colon  C(X, M_m)^\sim\rightarrow C(X, M_m)\oplus \K
\;\;\; \text{by }(a, \lambda)\mapsto (a+\lambda e_X, \lambda),$$ 
$$\Phi_Y\colon  C(Y, M_m)^\sim\rightarrow C(Y, M_m)\oplus \K
\;\;\; \text{by }(a, \lambda)\mapsto (a+\lambda e_Y, \lambda),$$ 
and
$$\Phi\colon  \widetilde{M_m}\rightarrow M_m\oplus \K
\;\;\; \text{by }(a, \lambda)\mapsto (a+\lambda e, \lambda).$$ 

Define $\widetilde R\colon C(X, M_m)^\sim\rightarrow C(Y, M_M)^\sim$ by 
$R(a+\lambda 1_X)=a|_Y+\lambda 1_Y,$ and 
define $R\colon C(X, M_m)\rightarrow C(Y, M_m)$ by
$R(a)=a|_Y.$
Then for every $x\in X$ and every $y\in Y,$ 
we have the following commutative diagram:

\begin{displaymath}
\begin{array}{ccccccc}
\widetilde{M_m}
&\xleftarrow{\widetilde{{\mathrm{ev}}_x}}
&C(X, M_m)^\sim   
&\xrightarrow{\widetilde R}
& C(Y, M_m)^\sim
&\xrightarrow{\widetilde{{\mathrm{ev}}_y}}
&\widetilde{M_m}\\
\downarrow \Phi
&&\downarrow \Phi_X
&&\downarrow \Phi_Y
&&\downarrow \Phi \\
M_m\oplus \K
&\xleftarrow{{\mathrm{ev}}_x\oplus id}
& C(X, M_m)\oplus \K
&\xrightarrow{R\oplus \K}
&C(Y, M_m)\oplus \K
&\xrightarrow{{\mathrm{ev}}_y\oplus id}
&M_m\oplus \K
\end{array}
\end{displaymath}

Now, since for all $x\in X,$ we have $$\tau(\widetilde
p(x))=\tau(\chi_\alpha(|\widetilde a(x)|)) =\chi_\alpha(\tau(|\widetilde
a(x)|))=\chi_\alpha(|\tau(\widetilde a(x))|)=\chi_\alpha(1)=0,$$ we see that for
all $x\in X,$ 
$\widetilde p(x)=(p(x), 0)$ for some projection $p(x)\in X.$  Since $u_0\in
C(Y, M_m)^\sim,$ there exists some $w_0\in C(Y, M_m)$ and some unitary
$\mu\in \K$ such that $u_0=(w_0, \mu).$  Note that
(\ref{Eq3.[UnitizedLem3.3PhillpsRSHA2].1}) implies that

\begin{equation}
\label{Eq3.[UnitizedLem3.3PhillpsRSHA2].3}
|\mu-1|=\left\|\tau\Bigl[[u_0(x)|\widetilde a(x)|-\widetilde
a(x)][1-\widetilde p(x)]\Bigr]\right\|<\delta\leq \epsilon,\;\;\;\text{for
all }x\in X.
\end{equation}

Let $\widehat v_0=w_0+\mu e_Y,$ so that $\Phi_Y(u_0)=(w_0+\mu e_Y, \mu)=(\widehat
v_0, \mu).$
Since $\Phi_Y$ is an isomorphism, we have $\widehat v_0\in U_0(C(Y,
M_m)).$  Let $\widehat a=a+e_X,$ so $(\widehat a, 1)=\Phi_X(\widetilde a).$
Next we compute: for each $x\in Y,$ we have
\begin{align*}
&\Phi\left(\left[u_0(x)|\widetilde a(x)|-\widetilde a(x)\right][1-\widetilde
p(x)]\right)\\
&\hspace*{3em}\mbox{}=\left[\Phi(u_0(x))|\Phi(\widetilde a(x))|-\Phi(\widetilde
a(x))\right]\Phi[1-\widetilde p(x))\\
&\hspace*{3em}\mbox{}= 
\Bigl[(\widehat v_0(x),\mu)\cdot
(|\widehat a(x)|,1)-(\widehat a(x),1)\Bigr](e-p(x), 1)\\
&\hspace*{3em}\mbox{}= 
\Bigl[(\widehat v_0(x)|\widehat a(x)|,\mu)
-(\widehat a(x),1)\Bigr](e-p(x), 1)\\
&\hspace*{3em}\mbox{}= 
\Bigl(\widehat v_0(x)|\widehat a(x)|-\widehat a(x),\mu-1\Bigr) \cdot(e-p(x), 1)\\
&\hspace*{3em}\mbox{}= 
\left(\Bigl[\widehat v_0(x)|\widehat a(x)|-\widehat
a(x)\Bigr]\Bigl[e-p(x)\Bigr],\mu-1\right).
\end{align*}
Thus, since $\Phi$ is isometric, we obtain the following from
(\ref{Eq3.[UnitizedLem3.3PhillpsRSHA2].1})
\begin{equation}
\label{Eq3.[UnitizedLem3.3PhillpsRSHA2].4}
\left\|\Bigl[\widehat v_0(x)|\widehat a(x)|-\widehat
a(x)\Bigr]\Bigl[e-p(x)\Bigr]\right\| <\delta<\delta',
\;\;\;\text{for all }x\in Y.
\end{equation}

Now, let $\pi\colon M_m\oplus \K\rightarrow M_m$ be the standard map.  Then we
compute again: for every $x\in X,$ we have
\begin{align*}
p(x)&=\pi(p(x), 0)=\pi\circ\Phi(p(x), 0)=\pi\circ\Phi(\widetilde p(x))\\
&=\pi\circ\Phi(\chi_\alpha(|\widetilde a(x)|))
=\chi_\alpha(\pi\circ\Phi(|\widetilde a(x)|))\\
&=\chi_\alpha\bigl(|\pi\circ\Phi(\widetilde a(x))|\bigr)
=\chi_\alpha\bigl(|\pi\circ\Phi(a(x),1)|\bigr)\\
&=\chi_\alpha\bigl(|\pi(a(x)+e,1)|\bigr)
=\chi_\alpha\bigl(|\pi(\widehat a(x),1)|\bigr)\\
&=\chi_\alpha(|\widehat a(x)|).
\end{align*}
Also, we have $n\geq \rank(p(x))=\rank(\widetilde p(x))\geq d/2$ and $\|\widehat
a\|\leq M.$  Let $\widehat w_t=\pi(\Phi_Y(w_t))$ for each $t.$  Then
$t\mapsto \widehat w_t$ is a homotopy in $U(C(Y, M_m))$ from $\widehat
w_t=\pi(\Phi_Y((0, 1))=\pi(e_Y, 1)=e_Y,$ to $\widehat
w_1=\pi(\Phi_Y(u_0))=\pi(\widehat v_0, \mu)=\widehat v_0.$

Thus by the choice of $\delta',$ there exist $\widehat v\in
U_0(C(X, M_m))$ and a homotopy $t\mapsto \widehat v_t$ in $U(C(X, M_m))$ for
$e_X$ to $\widehat v$ such that
$\widehat v|_Y=\widehat v_0,$ $\widehat v_t|_Y= \widehat w_t,$ and 
\begin{equation}
\label{Eq3.[UnitizedLem3.3PhillpsRSHA2].5}
\left\|\Bigl[\widehat v(x)|\widehat a(x)|-\widehat
a(x)\Bigr]\Bigl[e-p(x)\Bigr]\right\| <\epsilon,
\;\;\;\text{for all }x\in X.
\end{equation}

Let $u=(\widehat v-\mu e_X, \mu).$  Then $\Phi_X(u)=\Phi(\widehat v-\mu e_X, \mu)
=(\widehat v, \mu).$  Since $$(\widehat v, \mu)\in U_0(C(X, M_m)\oplus \K),$$ and
since $\Phi_X$ is a *-isomorphism, we have $u\in U_0(C(X, M_m)^\sim).$
Also for all $x\in Y,$ we have 
\begin{align*}
u(x)&=(\widehat v(x)-\mu e, \mu)=(\widehat
v_0(x)-\mu e, \mu)\\
&=(w_0(x)+\mu e-\mu e, \mu)=(w_0(x), \mu)=u_0(x).
\end{align*}
Thus $u|_Y=u_0.$

Then for all $x\in X,$ we have
\begin{align*}
&\Phi\left(\bigl[u(x)|\widetilde a(x)|-\widetilde a(x)\bigr]\bigl[1-\widetilde
p(x)\bigr]\right)\\
&\hspace*{3em}\mbox{}=
\bigl[\Phi(u(x))|\Phi(\widetilde a(x))|-\Phi(\widetilde
a(x))\bigr]\Phi\bigl(1-\widetilde p(x)\bigr)\\
&\hspace*{3em}\mbox{}=
\bigl[(\widehat v(x), \mu)(|\widehat a(x)|, 1)-(\widehat a(x),
1)\bigr](e-p(x), 1)\\
&\hspace*{3em}\mbox{}=
\bigl[(\widehat v(x)|\widehat a(x)|-\widehat a(x), \mu-1)\bigr] (e-p(x), 1)\\
&\hspace*{3em}\mbox{}=
\left(\bigl[\widehat v(x)|\widehat a(x)|-\widehat a(x)\bigr]\bigl[e-p(x)\bigr],
\mu-1\right).
\end{align*}
Thus for all $x\in X,$ we have, by
(\ref{Eq3.[UnitizedLem3.3PhillpsRSHA2].3}), 
(\ref{Eq3.[UnitizedLem3.3PhillpsRSHA2].5}), 
 and the fact that $\Phi$ is
isometric, 
\begin{align*}
&\left\|\bigl[u(x)|\widetilde a(x)|-\widetilde a(x)\bigr]\bigl[1-\widetilde
p(x)\bigr]\right\|\\
&\hspace*{3em}\mbox{}=\left\|
\left(\bigl[\widehat v(x)|\widehat a(x)|-\widehat a(x)\bigr]\bigl[e-p(x)\bigr],
\mu-1\right)
\right\|\\
&\hspace*{5em}\mbox{}(\text{the norm above is now taken in $M_m\oplus \K$})\\
&\hspace*{3em}\mbox{}=\max\left\{\left\|\Bigl[\widehat v(x)|\widehat a(x)|-\widehat
a(x)\Bigr]\Bigl[e-p(x)\Bigr]\right\|, |\mu-1|\right\}\\
&\hspace*{3em}\mbox{}<\epsilon.
\end{align*}

Let $v_t=\Phi_X^{-1}(\widehat v_t, \tau(w_t)).$  Then $t\mapsto v_t$ is a
homotopy in $U(C(X, M_m)^\sim).$  For each $t$ and each $y\in Y,$ 
we have $\widehat v_t(y)=\widehat w_t(y),$ so we have $(\widehat v_t(y), \tau(w_t))=
(\widehat w_t(y), \tau(w_t)).$  So $$R\oplus id(\widehat v_t, \tau(w_t))=(\widehat w_t,
\tau(w_t))=\Phi_Y(w_t)$$ and  $$\Phi_Y(w_t)=R\oplus
id(\Phi_X(v_t))=\Phi_Y(\widetilde R(v_t)).$$  Thus $w_t=\widetilde R(v_t).$  So
$w_t|_{Y}=v_t.$   Also $v_0=\Phi_X^{-1}(e_X, 1)=1_X$ and
$v_1=\Phi_X^{-1}(\widehat v, \tau(u_0)))=\Phi_X^{-1}(\widehat v, \mu)=u.$
This finishes the proof.
\end{proof}

The next lemma will ``stabilize'' the above lemma, and will be the 
one that we will need.
\begin{lem}
\label{StablizedLem3.3Phillips}
Let $0<\epsilon<1$ and let $0<\alpha_1<\alpha_2<1.$  Let
$X$ be a compact Hausdorff space with $\dim(X)=d<\infty.$  Let $Y\subseteq
X$ be a closed subset.  Let $a\in C(X, \KO)$ and let $\widetilde
a=a+1\in C(X, \KO)^\sim.$   For all $x\in X,$ let
$p_1(x)=p_{\alpha_1}(|\widetilde a(x)|)$ and let $p_2(x)=p_{\alpha_2}(|\widetilde
a(x)|).$  Suppose that for all $x\in X,$ $\rank(p_1(x))\geq d/2.$  Then
there exists $\delta>0$ such that: if $u_0\in U_0(C(Y, \KO)^\sim)$ is a
unitary and $h_0\colon [0, 1]\rightarrow U(C(Y, \KO)^\sim)$ is a homotopy such that 
$h_0(0)=1,$ $h_0(1)=u_0,$ and
\begin{equation}
\label{Eq3.[StablizedLem3.3Phillips].1}
\left\|[u_0(x)|\widetilde a(x)|-\widetilde
a(x)][1-p_1(x)]\right\|<\delta\;\;\;\text{for all }x\in Y,
\end{equation}
then there exists a unitary $u\in U_0(C(X, \KO)^\sim)$ and a homotopy
$h\colon [0, 1]\rightarrow U(C(X, \KO)^\sim)$  such that $h(0)=1,$ 
$h(1)=u,$ that $h(t)|_Y=h_0(t)$ for all $t,$ that $u|_Y=u_0,$
and that
$$\left\|[u(x)|\widetilde a(x)|-\widetilde
a(x)][1-p_2(x)]\right\|<\delta\;\;\;\text{for all }x\in X.$$
\end{lem}

\begin{proof}
Let $\epsilon,$ $\alpha_1,$ $\alpha_2,$ $X,$ $Y,$ $a,$ $p_1,$ and $p_2$
satisfy the hypothesis of the lemma, and let $M=2\|\widetilde a\|.$   Note that
$M\geq \|\widetilde a\|\geq 1.$

First of all, it is clear that there exists some $c\in C(X, \KO)_{s.a.}$
such that $|\widetilde a|=c+1.$  Denote $c+1$ by $\widetilde c.$  Note that
$\|\widetilde c\|=\|\widetilde a\|,$ since $(\widetilde c)^2=(\widetilde a)^*(\widetilde a).$
Let
$\alpha'=\frac{\alpha_1+\alpha_2}{2},$ and for each $x\in X,$ let
$p'(x)=p_{\alpha'}(|\widetilde a(x)|).$  Note that for all $x\in X,$ we have 
$p_2(x)\geq p'(x)\geq p_1(x)\geq d/2,$ and so we have $$\rank(p_2(x))\geq
\rank(p'(x))\geq \rank(p_1(x))\geq d/2.$$

By Lemma \ref{RankBdnessSpProj}, there exists $n\in \N$ such that
$\rank(p_2(x))=\rank(p_{\alpha_2}(\widetilde c))\leq n$ for all $x\in X.$
Apply Lemma \ref{UnitizedLem3.3PhillpsRSHA2} to $\epsilon/(16M)>0,$
$1>\alpha'>0,$ $n,$ and $M,$ to get $\delta_1>0.$  Without loss of
generality, assume that $\delta_1<\epsilon/(16M).$  Apply  Corollary
\ref{SemiContinuityOfSpectrumProjOfC(X,K)} to $\delta_1/(4M)$ in place of
$\epsilon,$ $\alpha_1,$ $\alpha'$ in place of $\alpha_2,$ and $M,$ to get
$\sigma_1>0.$  Apply Corollary
\ref{SemiContinuityOfSpectrumProjOfC(X,K)} again to $\delta_1/(4M)$ in place
of $\epsilon,$ $\alpha'$ in place of $\alpha_1,$ $\alpha_2,$ and $M$ to get
$\sigma_2>0.$  Let $$\delta=\min\{\epsilon/(16M), \delta_1/(16M),
\sigma_1/(16M),
\sigma_2/(16M), \alpha_2/(16M)\}.$$  Now let $u_0\in U_0(C(Y, \KO)^\sim)$ be a
unitary such that (\ref{Eq3.[StablizedLem3.3Phillips].1}) holds, and let
$h_0\colon [0, 1]\rightarrow U(C(Y, \KO)^\sim)$ be a homotopy from $1$ to $u_0.$

For each $k\in \N,$ embed $M_k$ into $M_{k+1}$ in the standard, and embed
$M_k$ into $\KO$ in the standard way.  Then we have 
$\KO=\overline{\bigcup_{k\geq 1}M_k}$ and
$\widetilde{\KO}=\overline{\bigcup_{k\geq 1}\widetilde{M_k}},$ where the
adjoined identity of each $\widetilde{M_k}$ is the same as the adjoined
identity of $\widetilde{\KO}.$  We will use $1$ to denote the adjoined
identity of $\widetilde{\KO}$ and $\widetilde{M_k},$ for $k\geq 1.$  The
above embeddings give the embedding of $C(X, M_k)$ into $C(X, M_{k+1})$
and into then $C(X, \KO).$  Then $C(X, \KO)=\overline{\bigcup_{k\geq 1} C(X,
M_k)}$ and $C(X, \KO)^\sim=\overline{\bigcup_{k\geq 1} C(X,
M_k)^\sim}.$  Again, we assume that the adjoined identity of $C(X,
\KO)^\sim$ is the same as the adjoined identity of $C(X, M_k)^\sim$ for
every $k\geq 1.$  We will use $1_X$ to denote the adjoined identity of
$C(X, \KO)^\sim$ and $C(X, M_k)^\sim$ for all $k\geq 1.$  Similarly, we use
$1_Y$ to denote the adjoined identity of $C(Y, \KO)^\sim$ and $C(Y,
M_k)^\sim$ for all $k\geq 1.$   

Then, we can find some $m\in \N,$ some $b\in C(X, M_m),$ and some homotopy
$$f_0\colon [0, 1]\rightarrow U(C(Y, M_m)^\sim)$$
such that 
\begin{eqnarray}
\label{Eq3.[StablizedLem3.3Phillips].2}
&\|a-b\|<\delta/(8M),\;\; \|\widetilde a-\widetilde b\|<\delta/(8M),\;\;
\Bigl\| |\widetilde b|-\widetilde c\Bigr\|<\delta/(8M)\\
\label{Eq3.[StablizedLem3.3Phillips].3}
&\|\widetilde b\|\leq M\\
\label{Eq3.[StablizedLem3.3Phillips].4}
&f_0(0)=1 \text{ and }\|f_0-h_0\|<\delta/(8M),
\end{eqnarray}
where $\widetilde b=b+1.$ Let $v_0=f_0(1).$ Then $\|v_0-u_0\|<\delta/(8M).$
Let $b'\in C(X, M_m)_{s.a.}$ be such that $|\widetilde b|=b'+1.$ Then
$\|b'+1\|=\|\widetilde b\|\leq M.$
Then (\ref{Eq3.[StablizedLem3.3Phillips].2}) implies that
\begin{equation}
\label{Eq3.[StablizedLem3.3Phillips].5}
\|b'-c\|<\delta/(8M).
\end{equation}
For each $x\in X,$ let $q'(x)=p_{\alpha'}(|\widetilde b(x)|)$ and let 
$q_2(x)=p_{\alpha_2}(|\widetilde b(x)|).$
By the choice of $\sigma_1,$
which is greater than $\delta/(8M),$ we have
(the space $X,$ and elements $a$ and $b$  in
Corollary
\ref{SemiContinuityOfSpectrumProjOfC(X,K)} are taken to be $X,$ $c$ and $b',$
respectively)
\begin{equation}
\label{Eq3.[StablizedLem3.3Phillips].6}
\|p_1(x)q'(x)-p_1(x)\|<\delta_1/(4M)\text{ and }
\rank(p_1(x))\leq \rank(q'(x)),
\end{equation}
for all $x\in X.$
By the choice of $\sigma_2,$ we have (the space $X,$ and the elements $a$
and $b$  in Corollary \ref{SemiContinuityOfSpectrumProjOfC(X,K)} are taken
to be $X,$ $b'$ and $c,$ respectively)
\begin{equation}
\label{Eq3.[StablizedLem3.3Phillips].7}
\|q'(x)p_2(x)-q'(x)\|\leq \delta_1/(4M)\text{ and } \rank(q'(x))\leq 
\rank(p_2(x)),
\end{equation}
for all $x\in X.$
Then 
\begin{equation}
\label{Eq3.[StablizedLem3.3Phillips].8}
n\geq \rank(p_2(x))\geq \rank(q'(x))\geq \rank(p_1(x))\geq d/2.
\end{equation}

Now, by (\ref{Eq3.[StablizedLem3.3Phillips].2}), for all $x\in Y,$ 
we have
\begin{align*}
&\left\|\bigl[v_0(x)|\widetilde b(x)|-\widetilde b(x)\bigr]
-\bigl[u_0(x)|\widetilde a(x)|-\widetilde a(x)\bigr]\right\|\\
&\hspace*{3em}\mbox{}\leq\left\|v_0(x)|\widetilde b(x)|-u_0(x)|\widetilde a(x)|\right\|
+\|\widetilde b(x)-\widetilde a(x)\|\\
&\hspace*{3em}\mbox{}\leq
 \left\|v_0(x)|\widetilde b(x)|-v_0(x)|\widetilde a(x)|\right\|
+\Bigl\|v_0(x)|\widetilde a(x)|-u_0(x)|\widetilde a(x)|\Bigr\|
+\delta/(8M)\\
&\hspace*{3em}\mbox{}< 2\delta/(8M)+\delta/8\leq 3\delta/8.
\end{align*}

Also, by (\ref{Eq3.[StablizedLem3.3Phillips].6}), for all $x\in X,$  we have 
\begin{align*}
&\|(1-p_1(x))(1-q'(x))-(1-q'(x))\|\\
&\hspace*{3em}\mbox{}= \|1-q'(x)-p_1(x)+p_1q'(x)-1+q'(x)\|\\
&\hspace*{3em}\mbox{}= \|p_1(x)q'(x)-p_1(x)\|\\
&\hspace*{3em}\mbox{}< \delta_1/(4M).
\end{align*}
Then combining the above two calculations and 
(\ref{Eq3.[StablizedLem3.3Phillips].1}),  we have
\begin{align*}
&\left\|\bigl[v_0(x)|\widetilde b(x)|-\widetilde
b(x)\bigr]\bigl[1-q'(x)\bigr]\right\|\\
&\hspace*{3em}\mbox{}\leq
\left\|\bigl[v_0(x)|\widetilde b(x)|-\widetilde
b(x)\bigr]\bigl[1-p_1(x)\bigr]\bigl[1-q'(x)\bigr]\right\|\\
&\hspace*{3em}\mbox{}\hspace*{3em}\mbox{}+
\left\|\bigl[v_0(x)|\widetilde b(x)|-\widetilde
b(x)\bigr]\Bigl\{[1-q'(x)]-\bigl[1-p_1(x)\bigr]\bigl[1-q'(x)\bigr]
\Bigr\}\right\|\\
&\hspace*{3em}\mbox{}\leq
\left\|\bigl[v_0(x)|\widetilde b(x)|-\widetilde
b(x)\bigr]\bigl[1-p_1(x)\bigr]\right\|
+ 2M\left\|[1-q'(x)]-\bigl[1-p_1(x)\bigr]\bigl[1-q'(x)\bigr]
\right\|\\
&\hspace*{3em}\mbox{}<
\left\|
\left\{\bigl[v_0(x)|\widetilde b(x)|-\widetilde b(x)\bigr]
-\bigl[u_0(x)|\widetilde a(x)|-\widetilde a(x)\bigr]\right\}
\bigl[1-p_1(x)\bigr]
\right\|\\
&\hspace*{3em}\mbox{}
\hspace*{3em}\mbox{} +\left\|
\bigl[u_0(x)|\widetilde a(x)|-\widetilde a(x)\bigr]
\bigl[1-p_1(x)\bigr]
\right\|
+\delta_1/2\\
&\hspace*{3em}\mbox{}\leq
\left\|
\bigl[v_0(x)|\widetilde b(x)|-\widetilde b(x)\bigr]
-\bigl[u_0(x)|\widetilde a(x)|-\widetilde a(x)\bigr]
\right\| +\delta+\delta_1/2\\
&\hspace*{3em}\mbox{}< 3\delta/8+\delta+\delta_1/2<\delta_1
\end{align*}
for all $x\in Y.$
Then by the choice of $\delta_1$ (with $X,$ $Y,$ $m,$ $a$ $\widetilde p,$
$w_t,$ and $u_0$
in Lemma \ref{UnitizedLem3.3PhillpsRSHA2}
taken to be, respectively, $X,$ $Y,$ $m,$ $b,$ $q',$ $f_0$ and $v_0$), 
there exists a unitary $v\in U_0(C(X, M_m)^\sim)\subseteq U_0(C(X,
\KO)^\sim)$ and a homotopy $f\colon [0, 1]\rightarrow U(C(X, M_m)^\sim)
\subseteq U(C(X, \KO)^\sim),$ such that $f(0)=1,$ $f(1)=v,$ 
$f(t)|_Y=f_0(t)$ for all $t,$ and $v|_Y=v_0,$ and
that 
\begin{equation}
\label{Eq3.[StablizedLem3.3Phillips].9}
\left\|\bigl[v(x)|\widetilde b(x)|-\widetilde
b(x)\bigr]\bigl[1-q'(x)\bigr]\right\|<\epsilon/(16M),\;\;\;\text{for
all }x\in X.
\end{equation}

Since, by (\ref{Eq3.[StablizedLem3.3Phillips].4}),
$\|f_0-h_0\|<\delta/(8M),$ and since $f(t)|_Y=f_0(t)$ for all $t\in [0,
1],$ there exists $h\colon [0, 1]\rightarrow U(C(X, \KO)^\sim)$ such that
$h(0)=1,$ $h(t)|_Y=h_0(t)$ for all $t,$ and $\|h-f\|<\delta/(4M).$  Let
$u=h(1).$  Then $\|u-v\|<\delta/(4M),$ and $u|_Y=h_0(1)=u_0.$
By (\ref{Eq3.[StablizedLem3.3Phillips].2}), we have 
\begin{align*}
&\left\|
\bigl[u(x)|\widetilde a(x)|-\widetilde a(x) \bigr]-
\bigl[v(x)|\widetilde b(x)|-\widetilde b(x)\bigr]
\right\|\\
&\hspace*{3em}\mbox{}\leq
\left\|
u(x)|\widetilde a(x)|- v(x)|\widetilde b(x)|\right\|
+\left\|
\widetilde a(x)-\widetilde b(x)
\right\|\\
&\hspace*{3em}\mbox{}\leq
\left\| u(x)|\widetilde a(x)|-u(x)|\widetilde b(x)| \right\|
+\left\| u(x)|\widetilde b(x)|- v(x)|\widetilde b(x)| \right\|
+\delta/(8M)\\
&\hspace*{3em}\mbox{}<2\delta/(8M)+\delta/4\leq \delta/2,
\end{align*}
for all $x\in X.$
Also by (\ref{Eq3.[StablizedLem3.3Phillips].7}), we have 
$$\|[1-q'(x)][1-p_2(x)]-[1-p_2(x)]\|<\delta_1/(4M)$$
for all $x\in
X.$ 
Thus by the two estimates above and 
(\ref{Eq3.[StablizedLem3.3Phillips].9}), for all $x\in X,$ we have
\begin{align*}
&\left\|
\bigl[u(x)|\widetilde a(x)|-\widetilde a(x) \bigr]
\bigl[1-p_2(x) \bigr]
\right\|\\
&\hspace*{3em}\mbox{}\leq
\left\|
\bigl[u(x)|\widetilde a(x)|-\widetilde a(x)\bigr]
\bigl[1-q'(x) \bigr]
\bigl[1-p_2(x) \bigr]
\right\|\\
&\hspace*{3em}\mbox{}
\hspace*{3em}\mbox{}+\left\|
\bigl[u(x)|\widetilde a(x)|-\widetilde a(x)\bigr]
\bigl\{
\bigl[1-p_2(x) \bigr]
-\bigl[1-q'(x) \bigr]
\bigl[1-p_2(x) \bigr]
\bigr\}
\right\|\\
&\hspace*{3em}\mbox{}\leq
\left\|
\bigl[u(x)|\widetilde a(x)|-\widetilde a(x)\bigr]
\bigl[1-q'(x) \bigr]
\right\|
+2M\delta_1/(4M)\\
&\hspace*{3em}\mbox{}\leq
\left\|
\bigl\{
\bigl[u(x)|\widetilde a(x)|-\widetilde a(x)\bigr]-
\bigl[v(x)|\widetilde b(x)|-\widetilde b(x)\bigr]
\bigr\}
\bigl[1-q'(x) \bigr]
\right\|\\
&\hspace*{3em}\mbox{}
\hspace*{3em}\mbox{}+\left\|\bigl[v(x)|\widetilde b(x)|-\widetilde
b(x)\bigr]\bigl[1-q'(x)\bigr]\right\|
+2M\delta_1/(4M)\\
&\hspace*{3em}\mbox{}<
\left\|
\bigl[u(x)|\widetilde a(x)|-\widetilde a(x)\bigr]-
\bigl[v(x)|\widetilde b(x)|-\widetilde b(x)\bigr]
\right\|\\
&\hspace*{3em}\mbox{}
\hspace*{3em}\mbox{}+\epsilon/(16M)+2M\delta_1/(4M)\\
&\hspace*{3em}\mbox{}<
\delta/2+\epsilon/(16M)+2M\delta_1/(4M)<\epsilon.\\
\end{align*}
This finishes the proof.
\end{proof}

 Let $A,$ $B,$ and $C$ be $\Cstar$-algebras.  Let
$\phi\colon A\rightarrow C$ and $R\colon B\rightarrow C$ be *-homomorphisms.  Let
$D=\{(a, b)\in A\oplus B\colon \phi(a)=R(b)\}.$  If we unitize $A,$ $B,$ $C,$
$\phi$ and $R,$ and let $$E=\{((a, \lambda), (b, \mu))\in \widetilde A\oplus
\widetilde B\colon \widetilde\phi(a)=\widetilde R(b)\},$$  then $((a, \lambda), (b,
\mu))\in
E$ if and only if $(a, b)\in D$ and $\lambda=\mu.$  So the map
$E\rightarrow \widetilde D$ defined by $((a, \lambda), (b, \lambda))\mapsto ((a,
b), \lambda)$ is a *-isomorphism.
Thus, given a SRSH system $$\left(X_1, A^{(1)}, \left(X_i, X_i^{(0)}, \phi_i, R_i,
A^{(i)}\right)_{i=2}^n\right)$$ and $A=A^{(n)},$ we can inductively unitize all the
algebras and maps to obtain the unitized system
$$\left(X_1, \widetilde{A^{(1)}}, \left(X_i, X_i^{(0)}, \widetilde{\phi_i},
\widetilde{R_i}, \widetilde{A^{(i)}}\right)_{i=2}^n\right).$$  Then $(a_i,
\lambda_i)_{i=1}^n\in \widetilde A$ if and only if $(a_i)_{i=1}^n\in A$ and
$\lambda_1=\cdots =\lambda_n;$ and each element $((a_i)_{i=1}^n,
\lambda)\in \widetilde A$ can be uniquely written as $(a_i, \lambda)_{i=1}^n.$
Also, if $a\in \widetilde A$ and $x\in X_k$ for some $k,$ then $a=(a_i,
\lambda)_{i=1}^n$ for some $(a_1, \ldots, a_n)\in A,$ and we will use $a(x)$ to
denote $(a_k, \lambda)(x)=(a_k(x), \lambda).$

\begin{lem}
\label{PrepareForApproxPolarDecomp1}
Let $$\left(X_1, A^{(1)}, \left(X_i, X_i^{(0)}, \phi_i, R_i,
A^{(i)}\right)_{i=2}^n\right)$$ be a SRSH system and let $A=A^{(n)}.$ Let $Y$ be a compact
Hausdorff space and let $\phi\colon A\rightarrow C(Y, \KO)$ be a *-homomorphism
(not necessarily non-vanishing). Let $\widetilde \phi$ denote the unitization
of $\phi.$   Let $\epsilon>0,$ let $1>\alpha>0,$ let $a\in
A,$ and let $\widetilde a=a+1\in \widetilde A.$  Let $u\in U_0(\widetilde A)$ be a
unitary
such that 
for all $x\in \bigsqcup_{i=1}^n (X_i\setminus X_i^{(0)}),$ 
\begin{equation}
\label{Eq3.[PrepareForApproxPolarDecomp1].1}
\left\|
\bigl[u(x)|\widetilde a(x)|-\widetilde a(x)\bigr]
\bigl[1-p_\alpha(|\widetilde a(x)|)\bigr]
\right\|<\epsilon.
\end{equation}
Then $\widetilde \phi(u)\in U_0(\widetilde{C(Y, \KO)})$ and all $y\in Y,$ we have
\begin{equation}
\label{Eq3.[PrepareForApproxPolarDecomp1].2}
\left\|
\bigl[\widetilde\phi(u)(y)|\widetilde \phi(\widetilde
a)(y)|-\widetilde\phi(\widetilde a)(y)\bigr]
\bigl[1-p_\alpha(|\widetilde\phi(\widetilde a)(y)|)\bigr]
\right\|<\epsilon.
\end{equation}
\end{lem}

\begin{proof}
Let $H$ denote the separable infinite dimensional Hilbert space and let
$1$ denote the  identity of $B(H).$  We identify the $\widetilde{\KO}$ 
with $\KO\oplus (\K \cdot 1)$ using the map $(a, \lambda)\mapsto
a+\lambda\cdot 1.$  For any compact Hausdorff space $Z,$ let $1_Z$ denote
the identity of $C(Z, B(H)).$  We identify the algebra $C(Z, \KO)\oplus (\K \cdot 1_Z)$
as a subalgebra of $C(Z, B(H))$ using the map 
$(a, \lambda \cdot 1_Z) \mapsto a+\lambda \cdot 1_Z.$  Then we identify 
$\widetilde{C(Z,
\KO)}$ with $C(Z, \KO)\oplus (\K \cdot 1_Z)\subseteq C(Z, B(H))$  using the map
$(f, \lambda)\mapsto f+\lambda \cdot 1_Z.$

Let $$\left(X_1, A^{(1)}, \left(X_i, X_i^{(0)}, \phi_i, R_i,
A^{(i)}\right)_{i=2}^n\right)$$ be a SRSH system and let $A=A^{(n)}.$ Let $Y$ be a compact
Hausdorff space and let $\phi\colon A\rightarrow C(Y, \KO)$ be a *-homomorphism
(not necessarily non-vanishing). Let $\widetilde \phi$ denote the unitization
of $\phi.$   Let $\epsilon>0,$ let $1>\alpha>0,$ let $a\in
A,$ and let $\widetilde a=a+1\in \widetilde A.$  Let $u\in U_0(\widetilde A)$ be a
unitary that satisfies (\ref{Eq3.[PrepareForApproxPolarDecomp1].1})
for all $x\in \bigsqcup_{i=1}^n (X_i\setminus X_i^{(0)}).$
With the above identifications, we can treat $\widetilde A$ as a subalgebra of
$C(X, B(H))$ using the maps $(b, \lambda)\mapsto b+\lambda 1_X,$ where $X$
is the total space of $A,$ and then the identity of $\widetilde
A$ is $1_X.$ So every element in $\widetilde A$ can be uniquely written as 
$((a_1, \lambda 1_{X_1}), \ldots, (a_n, \lambda 1_{X_n})),$ where
$\lambda\in \K$ and $(a_1, \ldots, a_n)\in A.$  
Then for all $b+\lambda 1_X\in \widetilde A,$ we have 
$\widetilde\phi(b+\lambda 1_X)=\phi(b)+\lambda 1_Y.$

It is clear that $\widetilde \phi(u)\in U_0(C(Y, \KO)^\sim).$ 
Fix $y\in Y.$  If the map $A\rightarrow \KO$ defined by $b\mapsto
\phi(b)(y)$ is the zero map, then for all $b\in A,$ we have $\widetilde \phi(\widetilde
b)(y)=1=|\widetilde\phi(\widetilde a)(y)|,$
and so $p_{\alpha}(|\widetilde\phi(\widetilde a)(y)|)=p_{\alpha}(1)=0.$  Since
$u=(v, \mu)\in U_0(\widetilde A)$ satisfies
(\ref{Eq3.[PrepareForApproxPolarDecomp1].1}), we have 
$|\mu-1|<\epsilon,$
and then the left side of (\ref{Eq3.[PrepareForApproxPolarDecomp1].2})
reduces to 
$\|[\mu\cdot 1-1][1-0]\|=|\mu-1|<\epsilon.$  So we can assume that
the map $A\rightarrow \KO$ given by $b\mapsto \phi(b)(y)$ is not the zero map.

Let $(p_i)_{i=1}^m$ be the family of mutually orthogonal projections in
$B(H),$ 
let $(w_i)_{i=1}^m$ be the family of isometries in $B(H)$ and 
let $(x_i)_{i=1}^m$ be the family of elements of $\bigsqcup_{k=1}^n
(X_k\setminus X_k^{(0)})$ that satisfy the conclusion of Proposition
\ref{CharaterOfKMapsFromSRSHA}.  Let
$p_{m+1}=1-\sum_{i=1}^m p_i.$
Then $(p_i)_{i=1}^{m+1}$ is still a mutually orthogonal family of
projections. For all $b+\lambda 1_X\in \widetilde A,$ we have 
\begin{align*}
\widetilde\phi(b+\lambda 1_X)(y)
&=\phi(b)(y)+\lambda 1
=\sum_{i=1}^m w_ib(x_i)w_i^*+\lambda\sum_{i=1}^m p_i
 +\lambda p_{m+1}\\\
&=\sum_{i=1}^m w_ib(x_i)w_i^*
 +\lambda\sum_{i=1}^m w_iw_i^*+\lambda p_{m+1}\\
&=\sum_{i=1}^m w_i(b(x_i)+\lambda \cdot 1)w_i^*
 +\lambda p_{m+1}\\
&=\sum_{i=1}^m w_i(b+\lambda 1_X)(x_i)w_i^*
 +\lambda p_{m+1}.
\end{align*}
Let $v\in A$ and $\mu\in \K$ satisfy $v+\mu 1_X=u.$
Then 
\begin{equation}
\label{Eq3.[PrepareForApproxPolarDecomp1].3}
\widetilde\phi(u)(y)=\widetilde \phi(v+\mu 1_X)=\sum_{i=1}^m
w_iu(x_i)w_i^*+\mu p_{m+1}.
\end{equation}
Also, we have
\begin{equation}
\label{Eq3.[PrepareForApproxPolarDecomp1].4}
\widetilde\phi(\widetilde a)(y)=\widetilde \phi(a+ 1_X)
=\sum_{i=1}^m w_i\tilde a(x_i)w_i^*+p_{m+1}
\end{equation}
and 
\begin{equation}
\label{Eq3.[PrepareForApproxPolarDecomp1].5}
|\widetilde\phi(\widetilde a)(y)|=\widetilde\phi(|\widetilde a|)(y)
=\sum_{i=1}^m w_i|\widetilde a|(x_i)w_i^*+p_{m+1}
=\sum_{i=1}^m w_i|\widetilde a(x_i)|w_i^*+p_{m+1}.
\end{equation}

Then (\ref{Eq3.[PrepareForApproxPolarDecomp1].3}) and 
(\ref{Eq3.[PrepareForApproxPolarDecomp1].5}) give

\begin{equation}
\label{Eq3.[PrepareForApproxPolarDecomp1].7}
\widetilde\phi(u)(y)|\widetilde\phi(\widetilde a)(y)|
=\sum_{i=1}^m w_iu(x_i)|\widetilde a(x_i)|w_i^*+\mu p_{m+1}.
\end{equation}

Also, by Corollary \ref{FunCalcOfMultipleDirectSum}, we have

$$p_\alpha(|\widetilde\phi(\widetilde a)(y)|)
=p_\alpha\left(\sum_{i=1}^m w_i|\widetilde
a(x_i)|w_i^*+p_{m+1}\right)
=\sum_{i=1}^m p_\alpha(w_i|\widetilde
a(x_i)|w_i^*)+p_\alpha(p_{m+1}),$$
where the functional calculus in the last expression is taken in
$p_iB(H)p_i$ for $i \in \{1, \ldots, m+1\}.$  Now, for each $i \in \{1, \ldots, m\},$ 
the map $B(H)\rightarrow p_iB(H)p_i$ defined by
$T\mapsto w_iTw_i^*$ is a unital *-isomorphism,  so we have 
$p_\alpha(w_i|\widetilde a(x_i)|w_i^*)=w_ip_\alpha(|\widetilde a(x_i)|)w_i^*,$ where
the last functional calculus is now taken in $B(H).$  So we have
\begin{equation}
\label{Eq3.[PrepareForApproxPolarDecomp1].6}
p_\alpha(|\widetilde\phi(\widetilde a)(y)|)
=\sum_{i=1}^m
w_ip_\alpha(|\widetilde a(x_i)|)w_i^*, 
\end{equation}
(functional calculus on both sides is taken in $B(H),$ i.e.\ the identity
used in the functional calculus is $\id_H$ on both sides).

Note that 
(\ref{Eq3.[PrepareForApproxPolarDecomp1].1}) implies that
$|\mu-1|<\epsilon.$  Then from 
(\ref{Eq3.[PrepareForApproxPolarDecomp1].1}),
(\ref{Eq3.[PrepareForApproxPolarDecomp1].4}), 
(\ref{Eq3.[PrepareForApproxPolarDecomp1].7}), and
(\ref{Eq3.[PrepareForApproxPolarDecomp1].6}), we have
\begin{align*}
&\left\|
\bigl[\widetilde\phi(u)(y)|\widetilde \phi(\widetilde a)(y)|-\widetilde\phi(\widetilde a)(y)\bigr]
\bigl[1-p_\alpha(|\widetilde\phi(\widetilde a)(y)|)\bigr]
\right\|\\
&\hspace*{3em}\mbox{} =
\left\|
\Bigl[(\mu-1)p_{m+1}+
\sum_{i=1}^m w_i\bigr[u(x_i)|\widetilde a(x_i)|
-\widetilde a(x_i)\bigr]w_i^*
\Bigr]\Bigr.\right.\\
&\hspace*{3em}\hspace*{3em}\mbox{}\cdot
\left.\Bigl.\Bigl[p_{m+1}+
\sum_{i=1}^m
w_i\bigl[
1-p_\alpha(|\widetilde a(x_i)|)
\bigr]w_i^*
\Bigr]
\right\|\\
&\hspace*{3em}\mbox{} =
\left\|
(\mu-1)p_{m+1}+
\sum_{i=1}^m w_i\bigr[u(x_i)|\widetilde a(x_i)|
-\widetilde a(x_i)\bigr]
\bigl[
1-p_\alpha(|\widetilde a(x_i)|)
\bigr]w_i^*
\right\|\\
&\hspace*{3em}\mbox{} =\max\left(\{|\mu-1|\}\cup\left\{\left\|\bigr[u(x_i)|\widetilde a(x_i)|
-\widetilde a(x_i)\bigr]\bigl[1-p_\alpha(|\widetilde a(x_i)|) \bigr]
\right\|\colon 1\leq i\leq m\right\}\right)\\
&\hspace*{3em}\mbox{} < \epsilon.
\end{align*}
This estimate holds for all $y\in Y,$ so result follows.
\end{proof}

\begin{lem}
\label{PrepareForApproxPolarDecomp2}
Let $$\left(X_1, A^{(1)}, \left(X_i, X_i^{(0)}, \phi_i, R_i,
A^{(i)}\right)_{i=2}^n\right)$$ be a SRSH system, let $A=A^{(n)}$ and let $X$ be the
total space.  Suppose that $\dim(X)=d<\infty.$  
Let $1>\epsilon>0$ and  let $1>\alpha>0.$  Let $a\in A,$ and let
$\widetilde a=a+1\in \widetilde A.$  Suppose that for all $x\in X,$ we have
$\rank(p_{\alpha/2}(|\widetilde a(x)|))\geq d/2.$  Then there exists
 $u\in U_0(\widetilde A)$
such that for all $x\in X,$ we have
\begin{equation}
\label{Eq3.[PrepareForApproxPolarDecomp2].0}
\left\|\bigl[u(x)|\widetilde a(x)|-\widetilde
a(x)\bigr]\bigl[1-p_{\alpha}(|\widetilde a(x)|)\bigr]\right\|<\epsilon.
\end{equation}
\end{lem}

\begin{proof}
First of all, if we let $x_0\in X_1,$ let $X_1^{(0)}=X_0=\{x_0\},$ let
$R_1\colon  C(X_1, \KO)\rightarrow C(X_1^{(0)}, \KO)$ be the restriction map, 
let $\phi_1\colon C(X_0, \KO)\rightarrow C(X_1^{(0)}, \KO)$ be the identity map,
and let $A^{(0)}=C(X_0, \KO),$
then $$\left(X_0, A^{(0)}, \left(X_i, X_i^{(0)}, \phi_i, R_i, A^{(i)}\right)_{k=1}^n\right)$$ is
again a SRSH system that gives the same SRSHA as the original system.  This
change does not affect any of the hypotheses or the conclusion of the
lemma.  Thus without loss of generality, assume that $X_1$ is just one
point set, and so $A^{(1)}\cong \KO.$

Now suppose $$\left(X_1, A^{(1)}, \left(X_i, X_i^{(0)}, \phi_i, R_i,
A^{(i)}\right)_{i=2}^n\right),$$ where $X_1$ is a one-point set, $1>\epsilon>0,$
$1>\alpha>0,$ and $a\in A$ satisfy the hypothesis of the lemma.  
Write $a = (a_1, \ldots, a_n)$ with $a_k \in C(X_k, \KO)$ for 
$k \in \{ 1, \ldots n\}.$

Choose $\alpha_1, \alpha_2, \ldots, \alpha_n\in \R$ such that
$0<\alpha/2=\alpha_1<\cdots<\alpha_n=\alpha.$  Now we inductively pick
$\delta_1, \ldots, \delta_n>0.$ Let $\delta_n=\epsilon/2.$ Suppose that
$\delta_k>0$ is picked.  Note that $\dim(X_k)\leq \dim(X)=d,$ and that for
each $x\in X_k,$ we have 
$$\rank(p_{\alpha_{k-1}}(|\widetilde a_k(x)|))=\rank(p_{\alpha_{k-1}}(|\widetilde
a(x)|))\geq \rank(p_{\alpha/2}(|\widetilde a(x)|))\geq d/2.$$
So we can apply Lemma \ref{StablizedLem3.3Phillips}, with $\epsilon,
\alpha_1, \alpha_2,$ $X,$ $Y,$ and $a$ in Lemma 
\ref{StablizedLem3.3Phillips} respectively taken to be 
$\min\{\delta_k/2, \epsilon/(2^k)\},$ $\alpha_{k-1},$ $\alpha_k,$ $X_k,$
$X_k^{(0)},$ and $a_k,$ to obtain $\delta_{k-1}'.$  Set
$\delta_{k-1}=\min\{\delta_k/2, \delta_{k-1}'\}.$
Next we inductively choose $u_k\in C(X_k, \KO)^\sim$ for $k \in \{1, \ldots, n\},$ 
and homotopies
$h_k\colon [0, 1]\rightarrow U(C(X_k, \KO)^\sim$ for $k \in \{1, \ldots, n\},$  such that
\begin{eqnarray}
\label{Eq3.[PrepareForApproxPolarDecomp2].1}
&&h_k(0)=1, h_k(1)=u_k,\;\;\text{for } k \in \{1, \ldots, n\},\\
\label{Eq3.[PrepareForApproxPolarDecomp2].2}
&&(h_1(t), \ldots h_{k}(t))\in U(\widetilde{A^{(k)}}),\;\;
\text{for }t\in [0, 1]\\
\label{Eq3.[PrepareForApproxPolarDecomp2].3}
&& (u_1, \ldots, u_k)\in U_0(\widetilde{A^{(k)}}),\;\;\text{for }k \in \{1, \ldots, n\},\\
\label{Eq3.[PrepareForApproxPolarDecomp2].4}
&&\left\|\bigl[u_k(x)|\widetilde a_k(x)|-\widetilde
a_k(x)\bigr](1-p_{\alpha_k}(|\widetilde a_k(x)))\right\|<\delta_k,\;\;
\text{for all } x\in X_k.
\end{eqnarray}

For each $\xi=(\xi_1, \ldots, \xi_n)\in \widetilde A,$ we will use $\xi^{(k)}$
to denote the first $k$ entries of $\xi.$ Note that 
$(\xi_1, \ldots, \xi_k)\in \widetilde{A^{(k)}}.$
Since $X_1$ is just a one-point space, it is clear that there exists $u_1\in
U_0(\widetilde{A^{(1)}})$ and a homotopy $h_1\colon [0, 1]\rightarrow
U(\widetilde{A^{(1)}})$ such that $h_1(0)=1$ and $h_1(1)=u_1,$ and that 
(\ref{Eq3.[PrepareForApproxPolarDecomp2].1}), 
(\ref{Eq3.[PrepareForApproxPolarDecomp2].3}), and
(\ref{Eq3.[PrepareForApproxPolarDecomp2].4})
hold for $k=1.$
Suppose that $u_k$ and $h_k$ are chosen to satisfy 
(\ref{Eq3.[PrepareForApproxPolarDecomp2].1}), 
(\ref{Eq3.[PrepareForApproxPolarDecomp2].2}), 
(\ref{Eq3.[PrepareForApproxPolarDecomp2].3}), and 
(\ref{Eq3.[PrepareForApproxPolarDecomp2].4}).

Let $v=\widetilde\phi_{k+1}(u^{(k)}),$ where $u^{(k)}=(u_1, \ldots, u_k)\in
\widetilde{A^{(k)}},$ and define $$f_0\colon [0, 1]\rightarrow
U(C(X_{k+1}^{(0)}, \KO)^\sim)$$ by $f_0(t)=\widetilde\phi_{k+1}(h_1(t), \ldots,
h_k(t)).$  Then
$v\in U_0(C(X_{k+1}^{(0)}, \KO)^\sim)$ and $f_0$ is a homotopy in
$U(C(X_{k+1}^{(0)}, \KO)^\sim)$ from 1 to $v.$  Also, applying Lemma
(\ref{PrepareForApproxPolarDecomp1}) to $A^{(k)}$ in place of
$A,$ $X_{k+1}^{(0)}$ in place of $Y,$
$\phi_{k+1}$ in place of $\phi,$ $a^{(k)}$ in place of $a,$
$\delta_k$ in place of $\epsilon,$ $\alpha_k$ in place of $\alpha,$
and $u^{(k)}=(u_1, \ldots, u_k)$ in place of $u,$ we have 
$$\big\|\big[v(x)|\widetilde\phi(\widetilde
a^{(k)})(x))|-\widetilde\phi(\widetilde a^{(k)})(x)\big]\big[1-p_{\alpha_k}(|\widetilde\phi(\widetilde
a^{(k)}(x)|)\big]\big\|<\delta_k,$$
for all $x\in X_{k+1}^{(0)}.$   Since
$\widetilde\phi_{k+1}(\widetilde a^{(k)})=\widetilde R(\widetilde a_{k+1}),$ we have 
$$\big\|\big[v(x)|\widetilde
a_{k+1}(x)|-\widetilde a_{k+1}(x)\big]\big[1-p_{\alpha_k}(|\widetilde
a_{k+1}(x)|)\big]\big\|<\delta_{k},$$
for all $x\in X_{k+1}^{(0)}.$  Then by the choice of $\delta_k,$ there
exists $u_{k+1}\in U_0(C(X_{k+1}, \KO)^\sim)$ and a homotopy $h_{k+1}$ in
$U(C(X_{k+1}, \KO)^\sim)$ such that $h_{k+1}(0)=1,$ such that
$h_{k+1}(1)=u_{k+1},$ such that $h_{k+1}(t)|_{X_{k+1}^{(0)}}=f_0(t)$ for all
$t\in [0, 1],$ such that $u_{k+1}|_{X_{k+1}^{(0)}}=v,$ and such that 
$$\big\|\big[u_{k+1}(x)|\widetilde a_{k+1}(x)|-\widetilde a_{k+1}(x)\big]
\big[1-p_{\alpha_{k+1}}(|\widetilde a_{k+1}(x)|)\big]\big\|<\delta_{k+1},$$
for all $x\in X_{k+1}.$   It is clear that $(u_1, \ldots, u_k, u_{k+1})$ is
a unitary $A^{(k+1)},$ and that for each $t\in [0, 1],$ we have $$(h_1(t), \ldots,
h_k(t), h_{k+1}(t))\in U(C(X_{k+1}, \KO)^\sim).$$  Then $t\mapsto 
(h_1(t), \ldots, h_{k+1}(t))$ is a homotopy in $U(C(X_{k+1}, \KO)^\sim)$
from $1$ to $(u_1, \ldots, u_k).$  So $(u_1, \ldots, u_k)\in
U_0(\widetilde{A^{(k+1)}}).$  This completes the inductive step.

Now take $u=(u_1, \ldots, u_n).$  Since for all $k \in \{1,\ldots, n\}$ and for all
$x\in X_k,$ we have $1-p_{\alpha_k}(|\widetilde a(x)|)\geq
1-p_\alpha(|\widetilde a(x)|),$  and since
$\delta_1<\delta_2<\cdots<\delta_k<\epsilon,$ 
(\ref{Eq3.[PrepareForApproxPolarDecomp2].4}) implies 
(\ref{Eq3.[PrepareForApproxPolarDecomp2].0}).
This finishes the proof.
\end{proof}

As a consequence of the above lemma, the next proposition will give an approximate 
polar decomposition for elements $a$
in a SRSHA such that the dimension of the the eigenspaces of the small eigenvalues
of $|a(x)|$ is large enough.

\begin{prop}
\label{ApproxPolarDecomp}
Let $$\left(X_1, A^{(1)}, \left(X_i, X_i^{(0)}, \phi_i, R_i,
A^{(i)}\right)_{i=2}^n\right)$$ be a SRSH system, let $A=A^{(n)},$ and let $X$ be the
total space.  Suppose that $\dim(X)=d<\infty.$  
Let $1>\epsilon>0$ and let $1>\alpha>0.$  Let $a\in A,$ and let
$\widetilde a=a+1\in \widetilde A.$  Suppose that for all $x\in X,$ we have
$\rank(p_{\alpha/2}(|\widetilde a(x)|))\geq d/2.$  Then there exists $u\in U_0(\widetilde A)$
such that $\|u|\widetilde a|-\widetilde a\|<\epsilon+2\alpha.$
\end{prop}

\begin{proof}
Let $u$ be the unitary obtained using 
Lemma \ref{PrepareForApproxPolarDecomp2}.  Then for all $x\in X$ and 
all $\xi\in H,$ where $H$ is the underlying Hilbert space, we have 
\begin{align*}
&\|[u(x)|\widetilde a(x)|-\widetilde a(x)](\xi)\|\\
&\hspace*{3em}\mbox{} \leq
\big\|\big[u(x)|\widetilde a(x)|-\widetilde a(x)\big]\big(1-p_\alpha(|\widetilde a(x)|)(\xi)\big)\big\|\\
&\hspace*{3em}\mbox{}\hspace*{3em}\mbox{}+\big\|\big[u(x)|\widetilde a(x)|-\widetilde a(x)\big]p_\alpha(|\widetilde a(x)|)(\xi)\big\|\\
&\hspace*{3em}\mbox{} < \epsilon \|\xi\|+\|(|\widetilde a(x)|)p_\alpha(|\widetilde a(x)|)(\xi)\|
+\|\widetilde a(x)p_\alpha(|\widetilde a(x)|)(\xi)\|\\
&\hspace*{3em}\mbox{} \leq \epsilon \|\xi\|+2\alpha \|\xi\|.
\end{align*}
Thus $\|[u(x)|\widetilde a(x)|-\widetilde a(x)]\|\leq \epsilon+2\alpha$ for all
$x\in X.$  So $\|u|\widetilde a|-\widetilde a\|\leq \epsilon+2\alpha.$
\end{proof}

\begin{cor}
\label{ApproxByInv}
Let $$\left(X_1, A^{(1)}, \left(X_i, X_i^{(0)}, \phi_i, R_i,
A^{(i)}\right)_{i=2}^n\right)$$ be a SRSH system, let $A=A^{(n)},$ and let $X$ be the
total space.  Suppose that $\dim(X)=d<\infty.$  
Let $1>\epsilon>0.$  Let $a\in A$ and let
$\widetilde a=a+1\in \widetilde A.$  Suppose that for all $x\in X,$ we have
$\rank(p_{\epsilon/8}(|\widetilde a(x)|))\geq d/2.$  Then there exists $b\in
\widetilde A$ such that $b$ is invertible and $\|\widetilde a-b\|<\epsilon.$
\end{cor}

\begin{proof}
Apply Proposition \ref{ApproxPolarDecomp} to $A,$ $\epsilon/4$ in place of
$\epsilon,$ $\epsilon/4$ in place of $\alpha,$ and $a\in A,$ to obtain a
unitary $u\in U_0(\widetilde A)$ such that $\|u|\widetilde a|-\widetilde
a\|<\epsilon/4+\epsilon/2=3\epsilon/4.$  Let $b=u(|\widetilde a|+\epsilon/4).$
Then $b$ is invertible and $$\|b-\widetilde a\|\leq \Bigl\|b-u|\widetilde
a|\Bigr\|+\|u|\widetilde a|-\widetilde a\|<\epsilon/4+3\epsilon/4=\epsilon.$$
\end{proof}

\begin{lem}
\label{KernelOfNonInvElem}
Let $$\left(X_1, A^{(1)}, \left(X_i, X_i^{(0)}, \phi_i, R_i,
A^{(i)}\right)_{i=2}^n\right)$$ be a SRSH system, let $A=A^{(n)},$ and let $X$ be the
total space.  Let $a\in A$ and let $\widetilde a=a+1\in \widetilde A.$ Let
$1>\alpha>0.$  Then the set $U=\{x\in X\colon \rank(p_\alpha(|\widetilde a(x)|)\geq
1\}$ is open.  Further, if $U\neq \varnothing,$ then $I_U=\{a\in
A\colon a|_{U^c}=0\}$ is a non-zero ideal of $A.$
\end{lem}

\begin{proof}
If $U=\{x\in X\colon \rank(p_\alpha(|\widetilde a(x)|))\geq 1\}$ is empty, then we
are done.  So assume that $U\neq\varnothing.$  To show that $U$ is open, it
is enough to show that every $x\in U$ is an interior point, i.e.\ there
exists some open $V\subseteq U$ such that $x\in V.$  Fix $x_0\in U.$

Apply Lemma \ref{SemiContOfSpProj2} to $\alpha$ and $|\widetilde a(x_0)|$ to
obtain
$\delta>0.$  The map $x\mapsto |\widetilde a(x)|$ is continuous, and the set
$V=\left\{x\in X\colon  \Bigl\| |\widetilde a(x)|-|\widetilde
a(x_0)|\Bigr\|<\delta \right\}$ is open and contains $x_0.$  If $x\in
V,$ then the choice of $\delta$ implies that $1\leq \rank(p_\alpha(|\widetilde
a(x_0)|))\leq \rank(p_\alpha(|\widetilde a(x)|)).$  Therefore $V\subseteq U,$
and hence $U$ is open.

To show that $I_U\neq 0,$ we verify the condition in part 1 of 
Lemma \ref{SRSHAINon0dealOpenSet}.  For each $k \in \{1, \ldots, n\},$ let
$U_k=X_k\cap U,$ and for each $k=2, \ldots, n,$ let
$$W_k=\left\{x\in X_k^{(0)}\colon  \Sp_x(\phi_k)\cap
\left(\bigsqcup_{i=1}^{k-1} U_i\right)\neq\varnothing\right\}.$$
Let $2\leq k\leq n$ and let $x\in W_k.$
Then $\Sp_x(\phi_k)\cap U\neq \varnothing,$ so let $y_0\in \Sp_x(\phi_k)\cap
U.$ Let $w_1, \ldots, w_l$ be the family of isometries with orthogonal
ranges such that 
$\phi_k(f)=\sum_{i=1}^l w_if(y_i)w_i^*$ for all
$f\in A^{(k-1)},$ where $y_i\in \Sp_x(\phi_k)$ for $i \in \{1, \ldots, l\}.$
Let $i_0$ be an integer such that $1\leq i_0\leq l$ and  $y_{i_0}=y_0.$  
Let $c\in A_{s.a.}$ be such that $|\widetilde a|=c+1.$  
Then 
\begin{align*}
p_\alpha(|\widetilde a(x)|)&=p_\alpha(c(x)+1)=p_{\alpha-1}(c(x))\\
&\hspace*{3em}\mbox{}=\sum_{i=1}^l
w_ip_{\alpha-1}(c(y_i))w_i^*\geq
w_{i_0}p_{\alpha-1}(c(y_0))w_{i_0}^*\\
&\hspace*{3em}\mbox{}=w_{i_0}p_\alpha(c(y_0)+1)w_{i_0}^*
=w_{i_0}p_\alpha(|\widetilde a(y_0)|)w_{i_0}^*.
\end{align*}
So, since $y_0\in U,$ we have
$\rank(p_\alpha(|\widetilde
a(x)|))\geq \rank(p_\alpha(|\widetilde a(y_0)|))\geq 1.$  Hence $x\in U_k,$
and so $x\in U_k\cap X_k^{(0)}.$  Therefore $W_k\subseteq U_k\cap
X_k^{(0)}.$

Now let $x\in U_k\cap X_k^{(0)}.$  Let $w_1, \ldots, w_l$ be the family of
isometries with orthogonal ranges such that 
$\phi_k(f)=\sum_{i=1}^l w_if(y_i)w_i^*$ for all
$f\in A^{(k-1)},$ where $y_i\in \Sp_x(\phi_k)$ for all $i \in \{1, \ldots, l\}.$
Then $$\rank(p_\alpha(|\widetilde a(x)|))
=\rank\left(\sum_{i=1}^l w_ip_\alpha(|\widetilde a(y_i)|)w_i^*\right)
=\sum_{i=1}^l \rank(p_\alpha(|\widetilde a(y_i)|)).$$  Since $x\in U,$ for some
$i \in \{1, \ldots, l\},$ we have $\rank(p_\alpha(|\widetilde a(y_i)|))\geq 1.$
Thus
$y_i\in \bigsqcup_{j=1}^{k-1} U_j.$  So $\Sp_x(\phi_k)\cap
\left(\bigsqcup_{j=1}^{k-1} U_j\right)\neq \varnothing,$ and so
$x\in W_k.$  Hence $U_k\cap X_k^{(0)}\subseteq W_k.$  

Thus by Lemma \ref{SRSHAINon0dealOpenSet}, $I_U\neq 0.$
\end{proof}

\begin{lem}
\label{RankOfSpProjInSimpleInductLimOfSRSHA}
Let $(A_n, \psi_n)$ be an inductive system of SRSHAs and let $A$ be the
inductive limit.  Let $X_n$ be the total space for $A_n.$
Suppose that $\psi_n$ is injective for all $n,$ 
that $\psi_n$ is non-vanishing for all $n,$ and suppose that $A$ is
simple.  Let $1>\alpha>0.$  Then for all $n\geq 1$ and all $a\in A_n$ such
that $\widetilde a=a+1$ is not
invertible in $\widetilde A_n,$ there exists some $m\geq n$ such that for all $k\geq
m$ and all $x\in X_k,$ we have $\rank(p_\alpha(|\widetilde \psi_{n, k}(\widetilde
a)(x)|))\geq 1,$ where $\widetilde \psi_{n, k}$ is the unitization of the map $\psi_{n,
k}.$
\end{lem}

\begin{proof}
Let $U=\{x\in X_n\colon  \rank(p_\alpha(|\widetilde a(x)|))\geq 1\}.$  We first show
that $U\neq \varnothing$.  Since $\widetilde a$ is not invertible, there exists some
$x_0$ in the total space of $A_n$ such that $\widetilde a (x_0)$ is not invertible. 
Then by the Fredholm Alternative, the operator $\widetilde a(x_0)$ is not injective,
which implies that $|\widetilde a(x_0)|$ is not 
injective.  Then $p_\alpha (|\widetilde a(x_0)|) \neq 0$, which implies that $x_0 \in
U$.  This shows that $U\neq \varnothing$.

By Lemma
\ref{KernelOfNonInvElem}, $I_U=\{a\in A_n\colon  a|_{U^c}=0\}$ is a non-zero
ideal.  Then by Proposition \ref{SpectrumStructureSimplIndLimSRSHA}, there exists
$m\geq N$ such that for all $k\geq m,$ and for all $x\in X_k,$ we have
$\Sp_x(\psi_{n, k})\cap U\neq \varnothing.$  Let $k\geq m,$ let $x\in
X_k,$ and let $w_1, \ldots, w_l$ be the family of isometries with
orthogonal ranges such that $\psi_{n, k}(f)(x)=\sum_{i=1}^l
w_if(y_i)w_i^*$ for all $f\in A_n,$ where $\{y_i\colon i=1, \ldots,
l\}=\Sp_x(\psi_{n,  k}).$ Let $y_0\in \Sp_x(\psi_{n, k})\cap U$ and choose
$1\leq i_0 \leq l$ such that $y_{i_0}=y_0.$
Let $c\in (A_n)_{s.a.}$ be such that $|\widetilde a|=\widetilde c.$  Then
$|\widetilde\psi_{n, k}(\widetilde a)|=\widetilde\psi_{n, k}(|\widetilde
a|)=\widetilde\psi_{n, k}(\widetilde c)=\psi_{n, k}(c)+1.$ Thus 
\begin{align*}
\rank(p_{\alpha}(|\widetilde \psi_{n, k}(\widetilde a)(x)|))
&=\rank(p_\alpha(|\widetilde \psi_{n, k}(\widetilde a)|(x)))
 =\rank(p_\alpha(\psi_{n, k}(c)(x)+1))\\
&=\rank(p_{\alpha-1}(\psi_{n, k}(c)(x)))
=\sum_{i=1}^l\rank(p_{\alpha-1}(c(y_i)))\\
&\geq \rank(p_{\alpha-1}(c(y_{i_0}))
 =\rank(p_\alpha(c(y_0)+1))\\
&=\rank(p_\alpha(\widetilde c(y_0)))=\rank(p_\alpha(|\widetilde a(y_0)|))\geq 1.
\end{align*}
The last inequality above holds because $y_0\in U.$  
\end{proof}

\begin{thm}
\label{StableRankOfSimpleInductLimOfSRSHA}
Let $(A_n, \psi_n)$ be an inductive system of SRSHAs and let $A$ be the
inductive limit.  Let $X_n$ be the total space for $A_n.$
Suppose that $\psi_n$ is injective and non-vanishing for all $n,$ 
and suppose that $A$ is
simple.  Also assume that there exists $d\in \N$ such that 
$\dim (X_n)\leq d$ for all $n\geq 1.$
Then $A$ has topological stable rank one.
\end{thm}

\begin{proof}
We first show that an element of the form $b+1\in \widetilde A,$ where $b\in
A,$  can be
approxmiated arbitrarily closely by some invertible element in $\widetilde A.$

Let $b\in A,$  let $1>\epsilon>0,$ and let $\widetilde b=b+1.$  Let $n\geq 1,$
and let $a\in A_n$ satisfy $\|\widetilde\psi^n(\widetilde a)-\widetilde
b\|<\epsilon/2,$ where $\psi^n: A_n \rightarrow A$ is the standard map that comes with the
inductive limit.  
If $\widetilde a$ is invertible in $A_n,$ then $\widetilde\psi^n(\widetilde a)$ is
invertible in $\widetilde A,$ and we are done.
So assume that $\widetilde a$ is
not invertible in $A_n.$  Then by Lemma
\ref{RankOfSpProjInSimpleInductLimOfSRSHA}, using $\epsilon/16$ as $\alpha,$
 find some $m_1\geq n$ such that for all $k\geq m_1,$
$\rank(p_{\epsilon/16}(|\widetilde\psi_{n, k}(\widetilde a)(x)|))\geq 1$ for all
$x\in X_k.$

For each $n\geq 1,$ let $X_{n, 1}, \ldots, X_{n, l(n)}$ be the base spaces
of $A_n,$ let $X_{n, 2}^{(0)}, \ldots , X_{n, l(n)}^{(0)}$ be the
attaching spaces, and let $X_{n, 1}^{(0)}=\varnothing.$  
If for all $k \geq m_1,$ the set 
$\bigsqcup_{i=1}^{l(k)} (X_{k, i}\setminus X_{k, i}^{(0)})$ is a finite set, 
then for all $k\geq m_1$ the algebra $A_k$ is simply a finite direct sum of 
copies of $\KO.$  
This means that 
$A_k$ 
has topological stable rank one for all $k \geq m_1,$ which implies that $A$ has
topological stable rank one, and we are done.  So we can assume that there
exists some $m_2\geq m_1$ such that 
$\bigsqcup_{i=1}^{l(m_2)}
(X_{m_2, i}\setminus X_{m_2, i}^{(0)})$ is infinite.  Let $1\leq l\leq
l(m_2)$ be the largest integer such that $X_{m_2, l}\setminus X_{m_2,
l}^{(0)}$ is infinite.  Then $A_{m_2}$ is isomorphic to
$A_{m_2}^{(l)}\oplus \left(\bigoplus_{i=1}^{l'} \KO\right)$ for
some $l'\in \N\cup\{0\},$ via some isomorphism $$h\colon A_{m_2}\rightarrow
A_{m_2}^{(l)}\oplus \left(\bigoplus_{i=1}^{l'} \KO\right)$$ such that the
composition 
$A_{m_2}\xrightarrow{h}A_{m_2}^{(l)}\oplus \left(\bigoplus_{i=1}^{l'} \KO\right)
\rightarrow A_{m_2}^{(l)}$ (the map on the right is the standard projection) 
is the restriction map $A_{m_1}\rightarrow
A_{m_1}^{(l)}.$
Let $d_1$ be
an integer greater that $d/2$ and let $x_1, \ldots, x_{d_1}\in X_{m_2, l}
\setminus X_{m_2, l}^{(0)}.$  For each $i \in \{1, \ldots,
d_1\},$ let $V_i\subseteq X_{m_2, l}\setminus X_{m_2, l}^{(0)}$ be an open
neighborhood of $x_i$ such
that $\{V_i\colon i=1,
\ldots, d_1\}$ is  disjoint.  For each $i \in \{1, \ldots, d_1\},$ let
$$J_i=\{a\in A_{m_2}^{(l)}\colon a|_{V_i^c}=0\}.$$  Then each $J_i$ is a non-zero
closed
two sided ideal of $A_{m_2}^{(l)}\oplus \left(\bigoplus_{i=1}^{l'}
\KO\right).$  For each $i \in \{1, \ldots, d_1\},$ let $I_i=h^{-1}(J_i).$  Since
$\{J_i\colon i=1, \ldots, d_1\}$ is  orthogonal, so is $\{I_i\colon i=1,
\ldots, d_1\}.$  For each $i \in \{1, \ldots, d_1\},$ let 
$$W_i=\{x\in X_{m_1}\colon {\text{ 
there exists some }} a\in I_i {\text{ such that }} a(x)\neq 0\}.$$
Then for each $i=1,
\ldots, d_1,$ we have $V_i\subseteq
W_i$ and  $W_i\cap \left(\bigsqcup_{j=1}^{l(m_2)}(X_{m_2, j}\setminus X_{m_2,
j}^{(0)})\right)=V_i.$

Now, for each $i \in \{1, \ldots, d_1\},$ apply Proposition
\ref{SpectrumStructureSimplIndLimSRSHA}, to obtain some $n_i\geq m_2$ such
that for all $k\geq n_i,$ and for all $x\in X_k,$ $\Sp_x(\psi_{m_1, k})\cap
W_i\neq \varnothing.$  Let $n_0=\max\{n_1, \ldots, n_{d_1}\}.$  Let $k\geq
n_0$ and let $x\in X_k.$  Then $\Sp_x(\psi_{m_2, k})\cap W_i\neq \varnothing$
for each $i \in \{1, \ldots, d_1\}.$  
So for each $i \in \{1, \ldots, d_1\},$ we can choose $y_i\in
\Sp_x(\psi_{m_2, k})\cap W_i.$  Since for each $i \in \{1, \ldots, d_1\},$ 
$$y_i\in W_i\cap \left(\bigsqcup_{i=1}^{l(m_2)}(X_{m_2, i}\setminus X_{m_2,
i}^{(0)})\right) = V_i,$$ and since $V_1, \ldots, V_{d_2}$ are pairwise disjoint, we
see that  $y_1,\ldots,  y_{d_1}$ are  distinct.
Let $w_1,
\ldots, w_t$ be isometries with mutually orthogonal ranges such that for
all $f\in A_{m_2}$ we have $\psi_{m_2, k}(f)(x)
=\sum_{i=1}^t w_if(z_i)w_i^*,$ where $\{z_i\colon i=1, \ldots,
t\}=\Sp_x(\psi_{m_2, k}).$  
Since $m_2\geq m_1,$ we have 
$\rank(p_{\epsilon/16}(|\widetilde\psi_{n, m_2}(\widetilde a)(y_i)|))\geq 1$ for
each
$i \in \{1, \ldots, d_1\}.$  Let $c\in (A_{m_2})_{s.a.}$ satisfy $|\widetilde\psi_{n,
m_2}(\widetilde a)|=\widetilde c.$ 
Then 
\begin{align*}
\rank(p_{\epsilon/16}(|\widetilde \psi_{n, k}(\widetilde a)(x)|)
&=\rank(p_{\epsilon/16}(\widetilde \psi_{m_2, k}(|\widetilde \psi_{n, m_2}(\widetilde
a)|)(x)))\\
&=\rank(p_{\epsilon/16}(\widetilde \psi_{m_2, k}(\widetilde c)(x)))\\
&=\rank(p_{(\epsilon/16)-1}(\psi_{m_2, k}(c)(x)))\\
&=\rank\left(p_{(\epsilon/16)-1}\left( \sum_{i=1}^t
w_ic(z_i)w_i^* \right)\right)\\
&\geq \sum_{i=1}^{d_2} \rank(p_{(\epsilon/16)-1}(c(y_i)))\\
&=\sum_{i=1}^{d_2} \rank(p_{\epsilon/16}(\widetilde c(y_i)))\\
&=d_1\geq d/2\geq \dim(X_k)/2.
\end{align*}
Then by Corollary \ref{ApproxByInv}, there exists some invertible element
$c\in \widetilde A_k$ such that $\|\widetilde\psi_{n, k}(\widetilde
a)-c\|<\epsilon/2.$  So $\widetilde\psi^k(c)$ is invertible in
$\widetilde A,$ and 
\begin{align*}
\|\widetilde\psi^k(c)-\widetilde b\|
&\leq\|\widetilde\psi^k(c)-\widetilde \psi^k(\widetilde\psi_{n, k}(\widetilde a))\|
+\|\widetilde \psi^k(\widetilde\psi_{n, k}(\widetilde a))-b\|\\
&=\|c-\widetilde\psi_{n, k}(\widetilde a)\| +\|\widetilde \psi^n(\widetilde a)-b\|\\
&<\epsilon/2 + \epsilon/2.
\end{align*}
Thus we have shown that for all $b\in A$ and all $\epsilon>0,$ there
exists some invertible element $c\in \widetilde A$ such that $\|\widetilde
b-c\|<\epsilon.$
Next will show that for all $b\in A$ and  all $\epsilon>0,$ there
exists some $c\in A$ such that $c+1$ is invertible and
$\|\widetilde c-\widetilde b\|<\epsilon.$  

Let $b\in A$ and let $1>\epsilon>0.$ 
By what we just proved above, $\widetilde b\in \overline{\text{inv}(\widetilde
A)},$ where $\text{inv} (\widetilde A)$ denote the set of all invertible elements
of $\widetilde A.$  
So there exists a sequence $(a_n, \lambda_n)\in \text{inv}(\widetilde A)$
such that $\|(a_n, \lambda_n)-(b, 1)\|\rightarrow 0.$  Then
$\lambda_n\rightarrow 1.$  So $(\lambda_n^{-1}a_n, 1)=\lambda_n^{-1}(a_n,
\lambda_n)\rightarrow \widetilde b.$  Thus we can pick some $n$ such that
$\|(\lambda_n^{-1}a_n, 1)-\widetilde b\|<\epsilon.$ Setting $c=\lambda_n^{-1}a_n,$
we see that $\widetilde c=\lambda_n^{-1}(a_n, \lambda_n)$ is invertible and
$\|\widetilde c-\widetilde b\|<\epsilon.$
Then by Proposition 4.2 of \cite{DefTSR1}, the algebra $A$ has topological stable
rank one.
\end{proof}

Many arguments in this chapter may be simplified greatly if every SRSHA is the 
tensor product of a RSHA with $\KO;$ however  we were not able to determine 
whether every SRSHA is the tensor product of a RSHA with $\KO.$  In the approach
we used when trying to resolve this question, we found that in order to show that a SRSHA is the 
tensor product of a RSHA with $\KO$, we needed to extend projection valued 
functions over a closed subspace of a compact metric space to the entire space.
This cannot be done in general, and so we feel that it is not true that every
SRSHA is the tensor product of a RSHA with $\KO.$  

Also, SRSHAs are likely to be
$\KO$-stable.  If $A$ is a SRSHA, then $A$ is contained in 
$B = \bigoplus_{i = 1}^n C(X_i, \KO)$ as a $\Cstar$-subalgebra, which implies that $A\otimes \KO$
is a $\Cstar$-subalgebra of $B \otimes \KO.$  The obvious
*-isomorphism from $B \otimes \KO$ to $B$ restricted to $A \otimes \KO$ may very well be a
*-isomorphism from $A \otimes \KO$ to  $A$.

	
	\renewcommand{\bibname}{\textsc{REFERENCES}} 
{\singlespacing
\bibliographystyle{elsart-num-sort}	
	\addcontentsline{toc}{chapter}{\bibname}

}					
	

\end{document}